\documentclass[11pt]{article}

    \usepackage{float}

    \usepackage{pgffor}
    \usepackage{tikz}
    \usetikzlibrary{calc,3d,arrows,positioning,shapes.misc}
	\usepackage{pgfplotstable}
	
    \usepackage[english]{babel}
    \usepackage{amsthm}
    \usepackage{amsmath}
    \usepackage{amssymb}
    \usepackage{inputenc}
    \usepackage{setspace}
    \usepackage{enumerate}
    \usepackage{esint}

    \usepackage{hyperref}

    \usepackage{xcolor}
    \usepackage{dsfont}
    \usepackage{array}
    \usepackage{graphicx}
    \usepackage[babel]{csquotes}
    \usepackage{geometry}
    \usepackage{bbm}
    \usepackage{stmaryrd}
    \usepackage[all]{xy}
    \usepackage{mathrsfs}

    \theoremstyle{plain}
    \newtheorem{thm}{Theorem}[section]
    \newtheorem{prop}[thm]{Proposition}
    
    \newtheorem{lem}[thm]{Lemma}
    \newtheorem{algo}[thm]{Algorithm}
    \theoremstyle{definition}
    \newtheorem{defn}[thm]{Definition}
    
    \newtheorem{rem}[thm]{Remark}

    \def \H{\mathscr{H}}
    
    \def \N{\mathbb{N}}
    \def \R{\mathbb{R}}
    \def \Z{\mathbb{Z}}
    
    \def \h{\sqrt{h}}

    \def \domain{{\mathbb{T}^3}}

    \newcommand {\supp} {\mathop \textup{supp}}
    \def \chara{\mathbf{1}}
    \def \eps{\varepsilon}

    \title{Analysis of thresholding for
codimension two motion by mean curvature: a gradient-flow approach}
    \author{Tim Laux\footnote{University of California, Berkeley, CA 94720-3840, USA. Please use {tim.laux@math.berkeley.edu} for correspondence.}
    \and Nung Kwan Yip\footnote{Purdue University, West Lafayette, IN 47907, USA}}
    \date{\today}
   
\begin{document}

      \maketitle
      
\newcommand{\nc}{\newcommand}

\nc{\Cal}[1]{{\mathcal {#1}}}
\nc{\Bf}[1]{{\bf {#1}}}
\nc{\Em}[1]{{\em {#1}}}
\nc{\Rm}[1]{{\rm {#1}}}
\nc{\nonu}{\nonumber}
\nc{\Headline}[1]{\noindent{\bf {#1}: }}

\nc{\EqnRef}[1]{(\ref{#1})}
\nc{\DefRef}[1]{\Bf{Definition \ref{#1}}}
\nc{\LemRef}[1]{\Bf{Lemma \ref{#1}}}
\nc{\ProRef}[1]{\Bf{Proposition \ref{#1}}}
\nc{\ThmRef}[1]{\Bf{Theorem \ref{#1}}}
\nc{\CorRef}[1]{\Bf{Corollary \ref{#1}}}
\nc{\SecRef}[1]{\Bf{Section \ref{#1}}}
\nc{\RemRef}[1]{\Bf{Remark \ref{#1}}}
\nc{\ChpRef}[1]{\Bf{Chapter \ref{#1}}}

\nc{\defref}[1]{Definition \ref{#1}}
\nc{\lemref}[1]{Lemma \ref{#1}}
\nc{\proref}[1]{Proposition \ref{#1}}
\nc{\thmref}[1]{Theorem \ref{#1}}
\nc{\corref}[1]{Corollary \ref{#1}}
\nc{\secref}[1]{Section \ref{#1}}
\nc{\remref}[1]{Remark \ref{#1}}
\nc{\chpRef}[1]{Chapter \ref{#1}}

\nc{\Proof}{\begin{proof}}
\nc{\EndProof}{\end{proof}}

\nc{\apriori}{{\it a priori }}
\nc{\Apriori}{{\it A priori }}
\nc{\Holder}{H\"{o}lder }

\nc{\Sty}{\displaystyle}
\nc{\MathSty}[1]{$\Sty{#1}$}
\def\Beqn#1\Eeqn{\begin{equation}#1\end{equation}}

\def\upchi{\raise2pt\hbox{$\chi$}}
\def\upnu{\raise0pt\hbox{$\nu$}}
\def\upstroke{\raise3pt\hbox{$|$}}
\nc{\Bdry}{\partial}
\nc{\Abs}[1]{\left|{#1}\right|}
\nc{\abs}[1]{|{#1}|}
\nc{\Ave}[1]{\bar{#1}}
\nc{\CurBrac}[1]{\left\{{#1}\right\}}
\nc{\curBrac}[1]{\{{#1}\}}
\nc{\Brac}[1]{\left({#1}\right)}
\nc{\brac}[1]{({#1})}
\nc{\SqrBrac}[1]{\left[{#1}\right]}
\nc{\sqrbrac}[1]{[{#1}]}
\nc{\BigBrac}[1]{\Big({#1}\Big)}
\nc{\bigBrac}[1]{\big({#1}\big)}
\nc{\BigCurBrac}[1]{\Big\{{#1}\Big\}}
\nc{\bigCurBrac}[1]{\big\{{#1}\big\}}
\nc{\BigSqrBrac}[1]{\Big[{#1}\Big]}
\nc{\bigSqrBrac}[1]{\big[{#1}\big]}
\nc{\BigAbs}[1]{\Big|{#1}\Big|}
\nc{\bigAbs}[1]{\big|{#1}\big|}
\nc{\converge}{\rightarrow}
\nc{\Converge}{\longrightarrow}
\nc{\WeakConverge}{\rightharpoonup}
\nc{\Del}{\triangle}
\nc{\Div}{\mbox{\rm{div}}}
\nc{\Equivalent}{\Longleftrightarrow}
\nc{\IntOverR}{\int_{-\infty}^{\infty}}
\nc{\IntOverRp}{\int_{0}^{\infty}}
\nc{\IntOverRm}{\int_{-\infty}^{0}}
\nc{\Imply}{\Longrightarrow}
\nc{\InnProd}[2]{\left\langle{#1},\,{#2}\right\rangle}
\nc{\innprod}[2]{\langle{#1},\,{#2}\rangle}
\nc{\BracInnProd}[2]{\left({#1},\,{#2}\right)}
\nc{\bracInnprod}[2]{({#1},\,{#2})}
\nc{\SqrinnProd}[2]{\left[{#1},\,{#2}\right]}
\nc{\sqrinnprod}[2]{[{#1},\,{#2}]}
\nc{\Lap}{\triangle}
\nc{\Lip}{\mbox{\rm{Lip}}}
\nc{\Lover}[1]{\frac{1}{#1}}
\nc{\Grad}{\nabla}
\nc{\MapTo}{\longrightarrow}
\nc{\Min}{\wedge}
\nc{\Max}{\vee}
\nc{\Norm}[1]{\left\|#1\right\|}
\nc{\norm}[1]{\|#1\|}
\nc{\SingleNorm}[1]{\left|#1\right|}
\nc{\singlenorm}[1]{|#1|}
\nc{\Pair}[2]{\left\langle{#1},\,{#2}\right\rangle}
\nc{\pair}[2]{\langle{#1},\,{#2}\rangle}
\nc{\DD}[1]{\frac{d}{d{#1}}}
\nc{\PDD}[1]{\frac{\partial}{\partial{#1}}}
\nc{\Spt}{\mbox{\rm{spt}}}
\nc{\spt}{\mbox{\rm{spt}}}
\nc{\SuchThat}{\ni}
\nc{\Sum}[2]{\sum_{#1}^{#2}}
\nc{\wtilde}[1]{\widetilde{#1}}
\nc{\Text}[1]{{\mbox{{\rm #1}}}}
\nc{\TextMath}[1]{\mbox{\,\,\,{#1}\,\,\,}}
\nc{\intersect}{\cap}
\nc{\Intersect}{\bigcap}
\nc{\union}{\cup}
\nc{\Union}{\bigcup}
\nc{\ColTwo}[2]
{\left(\begin{array}{c}{#1}\\{#2}\end{array}\right)}
\nc{\ColThree}[3]
{\left(\begin{array}{c}{#1}\\{#2}\\{#3}\end{array}\right)}
\nc{\ColFour}[4]
{\left(\begin{array}{c}{#1}\\{#2}\\{#3}\\{#4}\end{array}\right)}
\nc{\ColFive}[5]
{\left(\begin{array}{c}{#1}\\{#2}\\{#3}\\{#4}\\{#5}\end{array}
\right)}
\nc{\coltwo}[2]
{\begin{array}{c}{#1}\\{#2}\end{array}}
\nc{\colthree}[3]
{\begin{array}{c}{#1}\\{#2}\\{#3}\end{array}}
\nc{\colfour}[4]
{\begin{array}{c}{#1}\\{#2}\\{#3}\\{#4}\end{array}}
\nc{\colfive}[5]
{\begin{array}{c}{#1}\\{#2}\\{#3}\\{#4}\\{#5}\end{array}}

\nc{\Vol}[2]{\Bf{#1}({#2})}

\nc{\purple}[1]{\textcolor{purple}{#1}}
\nc{\red}[1]{\textcolor{red}{#1}}
\nc{\RED}[1]{\textcolor{red}{\bf #1}} 
\nc{\blue}[1]{\textcolor{blue}{#1}}
\nc{\BLUE}[1]{\textcolor{blue}{\bf #1}} 
\nc{\green}[1]{\textcolor{green}{#1}}
\nc{\GREEN}[1]{\textcolor{green}{\bf #1}} 


      \begin{abstract}
The Merriman-Bence-Osher (MBO) scheme, also known as thresholding or diffusion generated motion, is an efficient numerical algorithm for computing mean curvature flow (MCF). It is fairly 
well understood in the case of hypersurfaces. This paper establishes the first 
convergence proof of the scheme in codimension two.
We concentrate on the case of the curvature motion of a filament (curve) in 
$\mathbb{R}^3$. 
Our proof is based on a new generalization of the  minimizing movements 
interpretation for hypersurfaces (Esedoglu-Otto '15) by means of an energy that approximates the Dirichlet energy 
of the state function.
As long as a smooth MCF exists, we establish uniform energy estimates for the approximations away from the smooth solution and prove convergence towards 
this MCF.
The current result which holds in codimension two relies in a very 
crucial manner on a new sharp monotonicity formula for the thresholding energy. This is an improvement of an earlier approximate version.
	
     \medskip
 
 
     \medskip

     \end{abstract}
    
    
  \section{Introduction}
     
     \subsection{Motivation}
This paper is devoted to the analysis of the thresholding scheme in 
codimension two which may model the motion of vortices (points) in the plane, 
filaments (curves) in three-dimensional space or 
two-dimensional surfaces in four dimensions. 
For the sake of definiteness we will mostly focus on the---from our point of view---most relevant case of a curve in $\R^3$. 
Important applications of curvature-driven motion of filaments in $\R^3$ include superconductivity (magnetic flux tubes in type-I superconductors move by curve-shortening flow), fluid dynamics (where the motion of vortex lines is described by \emph{binormal} curvature flow), image processing (in particular for identifying vasculature in magnetic resonance angiography (MRA) images), and many more. Curiously, the curvature flow of a filament has also been used to define the curve-shortening 
flow of immersed planar curves past singularities
\cite{altschuler1992shortening}. 

\medskip

The analysis of motion by mean curvature for a hypersurface has a long history, 
starting from the fundamental work of Brakke \cite{brakke1978motion}. 
A range of techniques has
later been developed to further the understanding of such 
geometric evolutions. These include 
singular perturbations \cite{bronsard1991motion, 
chen1992generation, de1995geometrical, ilmanen1993convergence}, 
the level set formulation \cite{EvansSpruck1, chen1991uniqueness,
evans1992phase}, 
and variational time stepping or minimizing movements \cite{ATW93, LucStu95}. 
However,
for higher codimension curvature motions, there are relatively fewer results. 
We refer to the
works \cite{altschuler1992shortening, WangMuTaoClassical} for statements in the 
classical setting and \cite{WangMuTaoReview1, WangMuTaoReview2} for reviews
of the current status.
One reason is that the comparison principle which is used often in the 
hypersurface case is not applicable in higher codimension. 
However, variational techniques are quite versatile. 
The current paper uses a variational interpretation to 
analyze an efficient numerical scheme and prove its convergence to
motion by mean curvature of curves in three dimensional space.

\medskip

The idea of thresholding goes back to the 1992-paper \cite{MBO92} of Merriman, Bence and Osher treating the evolution of hypersurfaces by their mean curvature. The algorithm is henceforth often called the MBO scheme.
It is a two-step time discretization procedure easily described as:
Given an open set $\Omega_0$ of $\R^d$ and a time-step size $h>0$,
a sequence of open subsets $\CurBrac{\Omega_n}_{n\geq 1}$ of 
$\mathbb{R}^d$ is generated by alternating between
\begin{enumerate}
\item[1.] solving the linear heat equation for time $h$, starting from the
characteristic function of the set $\Omega_n$:
\Beqn\label{MBO-hyper-I}
\partial_t v = \Delta v,\quad\text{for $0< t < h$},\quad 
v(x,0) = \chi_{\Omega_n}(x) = \left\{
\begin{array}{ll}
1, & x\in \Omega_n,\\
0, & x\not\in \Omega_n,
\end{array}
\right.
\Eeqn
\end{enumerate}
and
\begin{enumerate}
\item[2.] projecting the function $v(x,h)$ onto $\CurBrac{0,1}$ to obtain the
new set $\Omega_{n+1}$:
\Beqn\label{MBO-hyper-II}
\begin{array}{l}
\Omega_{n+1} = \Big\{x: v(x,h) > \Lover{2}\Big\}.
\end{array}
\Eeqn
\end{enumerate}
In the following, we will use $u^n$ to denote $\chi_{\Omega_n}$, 
the state of the evolution at the $n$-th time step.
The second procedure above is also called the \Em{thresholding} step due to 
the use of the threshold value $\Lover{2}$. (Sometimes, the 
completely equivalent choice of $\CurBrac{-1,1}$-valued functions is used. 
In this case, $u^n$ and 
$\Omega_n$ are related by $u^n=2\chi_{\Omega_n}-1$. Then the threshold value 
is $0$ and the projection step 
above can be simply stated as $u^n=\frac{v}{\abs{v}}$.) The sequence of sets 
$\CurBrac{\Omega_n}_{n\geq 0}$ is shown to converge to
motion by mean curvature in the viscosity sense 
\cite{evans1993convergence, barles1995simple}. These proofs rely very much on
the comparison principle which is satisfied by the scheme above.
See also \cite{ishiiGeneralRadialKernel, ishii1999threshold}
for a generalization of the result for more general kernels.

\medskip

Thresholding for a filament in $\mathbb{R}^3$, 
due to Ruuth, Merriman, Xin, and Osher \cite{ruuth2001diffusion},
 is just as simple to describe as the hypersurface case. Consider 
an $\mathbb{R}^2$- or complex-valued function 
$u$ defined on $\mathbb{R}^3$ such that it has length one almost everywhere,
or equivalently a measurable function $u:\mathbb{R}^3\to\mathbb{S}^1$. 
Given a curve $\Gamma\subset\mathbb{R}^3$, it is fairly straightforward to 
construct a function $u$ such that it ``winds around'' $\Gamma$ with 
winding number equal to one, see Section \ref{rem Eu0} 
and Appendix \ref{appendix:u0}.
The curve is also the set where $u$ is ``singular''
(see Fig. \ref{fig:slices}). The thresholding scheme in this case is very similar to the one for hypersurfaces. In the first step, as in \eqref{MBO-hyper-I}, we diffuse the predecessor $u^{n-1}$, which is a unit vector field. The second step \eqref{MBO-hyper-II} is replaced by
projecting $v$ onto $\mathbb{S}^1$:
\Beqn
u^{n+1}(x)= \frac{v(x,h)}{\abs{v(x,h)}}.
\Eeqn

The main result of the present paper is the convergence of the above 
algorithm to the mean curvature flow of $\Gamma$. 
A heuristic argument, using asymptotic expansions is given in 
\cite{ruuth2001diffusion}. We will also briefly describe the underlying
formal computation in Appendix \ref{appendix asymptotic expansion}.
To the best of our knowledge, our work is the first convergence proof of 
the thresholding scheme in higher codimension.

\medskip

We spend a moment here to interpret the above thresholding scheme from the
point of view of Ginzburg-Landau functionals and their gradient flows.
These concepts appear often in the study of phase transition and interface
motions. The functional has the form
\Beqn\label{GLF}
\Cal{F}_\eps(u) =
\int_\Omega \frac12 \left| \nabla u\right|^2 +\frac1{\eps^2} W(u)\,dx
\Eeqn
where $u:\Omega\subseteq\mathbb{R}^n\to\mathbb{R}^m$ is the phase function
and
$W:\mathbb{R}^m\to\mathbb{R}_+$ is a (non-negative) potential function which
vanishes on some prescribed set. In the above, $\eps\ll 1$ 
is a small positive number. 
The gradient flow of $\Cal{F}_\eps$ (in the $L^2$-sense) is given by
\begin{equation}\label{GL}
\partial_t u_\eps = \Delta u_\eps  -\frac1{\eps^2}\nabla_u W(u_\eps)
\,\Brac{=-\frac{\partial\Cal{F}_\eps(u)}{\partial u}}.
\end{equation}
A direct computation gives the following energy dissipation law:
\Beqn\label{GL ED}
\frac{d}{dt}\Cal{F}_\eps(u_\eps)
= -\int\Abs{\partial_t u_\eps}^2\,dx.
\Eeqn
For both stationary and dynamic considerations, the singular limit
$\eps\to 0$ is one of the key questions to investigate.

\medskip

By setting different values for the dimensions of the ambient
space and the range, the functional can model various geometric objects. 
For example, to model hypersurfaces and their motions, 
 one may take $m=1$, i.e., $u$ is scalar-valued and $W(u)=(1-u^2)^2$.
In this case, $W$ vanishes on the discrete set $\CurBrac{-1, 1}$.
The energy is usually called the Cahn-Hilliard functional, whose dynamics \eqref{GL} are known as the Allen-Cahn equation due to their  first appearance in the materials science literature \cite{allen1979microscopic}.
The typical behavior is that the function $u_\eps$ will partition the ambient
space into two domains $\Omega_-$ and $\Omega_+$ on which $u_\eps$ takes values
roughly equal to $1$ and $-1$ separated by a narrow transition layer of width $ O(\eps)$. Hence in the limit $\varepsilon\to0$ this layer forms a sharp interface which can be described precisely as a minimal surface in the stationary regime or it evolves according to MCF in the dynamical case.
We refer to \cite{rubinstein1989fast} for a heuristic illustration
which has been proved rigorously in various mathematical settings---see the beginning of this introduction.
If $n=m=2$, the function $u_\eps$ is defined on (a subset of)
$\mathbb{R}^2$ and takes values in $\mathbb{R}^2$,
or equivalently is \emph{complex-valued}.
A common choice for the potential function is
$W(u)=(1-\abs{u}^2)^2$ so that the zero set of $W$ is 
the unit circle $\mathbb{S}^1$.
Hence any $u_\eps$ with reasonably low functional energy value 
$\Cal{F}_\eps(u_\eps)$ has point-wise norm approximately equal to one.
 In this case, by topological reasoning, 
$u_\eps$ can have points (vortices) as its singular (defect) sets. 
Such functionals are widely used in the modeling and analysis of vortices, 
their dynamics and interaction in superconductivity 
phenomena (see \cite{BBHBook, SerfatySandierBook}). Next, if we take $n=3$ and $m=2$,
i.e., $u_\eps$ is a complex valued function 
defined on (a subset of) $\mathbb{R}^3$, and the
same potential function $W(u)=(1-\abs{u}^2)^2$, then $u_\eps$ can 
incorporate curves as its singular sets, see Figure \ref{fig:slices}. 
This is also used in the modeling of vortex lines in superconductivity
as well as superfluids \cite{rubinsteinSelfInduced, rubinsteinPismen}. 
Even more generally,
the zero set of $W$ can consist of disjoint Riemannian manifolds. Then
the dynamics \eqref{GL} can model harmonic heat flows 
\cite{rubinsteinHarmonic, linPanWangHarmonic}. The above
description clearly demonstrates the range of applicability of the functional
\eqref{GLF} and explains the intensive mathematical activities surrounding it.
We defer to Section \ref{sec:related work} for more recent references 
of related work.

\medskip

Note that the gradient flow dynamics \eqref{GL} can be formally solved by
\emph{operator splitting}, alternating the following two steps:
\begin{eqnarray}
\text{(i) linear diffusion:}\,\, \partial_t u_\eps = \Delta u_\eps;
\quad
\text{(ii) fast reaction:}\,\, 
\partial_t u_\eps = -\Lover{\eps^2}\nabla_u W (u_\eps).
\end{eqnarray}
The key idea of \cite{MBO92} is to replace Step (ii) by instantly projecting
$u_\eps$ onto the zero set of $W$. Referring to the description at the 
beginning, we have that for the hypersurface case,
$u_\eps$ is projected onto $\CurBrac{-1,1}$ while for the filament case,
$u_\eps$ is projected onto $\mathbb{S}^1$. 
We remark that this projection 
step clearly generalizes to the case when $W$ vanishes on more general sets, for example
multiple disconnected copies of 
$\mathbb{S}^N$.

\medskip

Similar to the Ginzburg-Landau equation \eqref{GL}, MCF also has a gradient-flow structure. Indeed, it is the $L^2$-gradient flow of the area functional. This suggests to analyze the dynamics 
using variational methods. Such an approach has been implemented
in \cite{ATW93, LucStu95} for the MCF of hypersurfaces. 
De Giorgi \cite{de1993new} formalized this idea in a more general setting, which is now often called minimizing movements. We refer the reader to \cite{ambrosio2006gradient} for a more 
contemporary exposition.
The key idea of such a method is to discretize the evolution in time
(with time step $h>0$) and obtain the state $u^n$ at the $n$-th time 
step by minimizing the functional
\Beqn\label{MinMove}
E(u) + \frac1{2h}d^2(u,u^{n-1})
\Eeqn
where $E$ is the energy of the state, 
$d$ is the distance or metric compatible to the gradient structure of 
$E$, and $u^{n-1}$ is the state at the previous time step. 
The overall effect of minimizing \eqref{MinMove} is that the energy decreases according to some dissipation mechanism.
Furthermore, the sequence of minimizers $u^n$ formally satisfies
the implicit time discretization scheme for the gradient flow of $E$
\Beqn \label{EL abstract}
\frac{u^n - u^{n-1}}{h} = -\nabla E(u^n).
\Eeqn
The limit as $h\to 0$ of the sequence $\CurBrac{u^n}_{n\geq 0}$ thus obtained 
is then called a minimizing movements for $E$. We emphasize here that
the definition of the metric $d$ which provides appropriate
dissipation mechanism is just as important as the energy $E$ itself.
In fact, if the metric $d$ on the state space is the induced distance of some Riemannian metric (via shortest paths), then $\nabla E$ appearing in \eqref{EL abstract} is the gradient of the functional $E$ w.r.t.\ this Riemannian metric.

\medskip

The compatibility of the thresholding scheme to the above gradient-flow structure in form of a 
minimizing movements interpretation was first made by Esedo\u{g}lu and Otto in the work 
\cite{EseOtt14}. In the hypersurface case, they constructed an energy 
that approximates---or more precisely, $\Gamma$-converges to---the
interfacial area. Their approach also allows them to handle multi-phase
systems with a broad class of surface tensions. This generalization 
has been an open problem for several decades. Based on this minimizing movements principle,
the work \cite{laux2015convergence} provides a rigorous analysis of the scheme
in the dynamical setting and gives a convergence proof to motion by
mean curvature in the multi-phase case. 
We will comment more on these related works in \S \ref{sec:idea of proof} and \S \ref{sec:related work}.

\medskip

Interpreting thresholding as a minimizing movements scheme has practical implications as well.
We will see that in our case of codimension two, the minimizing movements principle furnishes a generalization of the scheme to incorporate Dirichlet or Neumann boundary conditions as well as a chemical potential leading to a pinning 
effect. 
An advantage of the current approach is that the same Gaussian kernel
works with very minor modification. Hence numerical efficiency is not
affected. In the next section, we will briefly describe the basis of our method of proof in the filament case.

\subsection{Idea of proof}\label{sec:idea of proof}
The protagonist in the present work is the approximate energy
\begin{equation}\label{def Eh}
E_h(u) = \frac1h \int \left(1-u\cdot G_h\ast u\right) dx
\end{equation}
defined for any unit vector fields $u\colon \R^3\to \mathbb{S}^1$; here $G_h$ denotes 
the heat kernel in $\R^3$ evaluated at time $h$, i.e., 
a Gaussian kernel of variance $2h$.
Note that the counterpart of \eqref{def Eh} in the hypersurface case \cite{EseOtt14} is given by
\Beqn
F_h(\chi) = \frac1\h \int \left(1-\chi\right) G_h\ast\chi \,dx, \quad \text{where }\chi \colon \R^3 \to \{0,1\}.
\Eeqn
There is a fundamental difference between them.
The energy $F_h$ measures the heat transfer from the set 
$\{\chi=1\}$ into its complement $\{\chi=0\}$ which is roughly equal to 
the $(d-1)$-dimensional area or measure of the boundary of $\{\chi=1\}$. 
This can be rigorously justified by that $F_h$ $\Gamma$-converges to 
$c_0\Cal{H}^{n-1}(\partial\{\chi=1\})$ for some constant $c_0$ 
\cite[Prop.\ A.1]{EseOtt14}.
On the other hand, the energy $E_h$ measures the distance of the diffused 
vector field $G_{h/2}\ast u$ to the sphere $\mathbb{S}^1$, i.e., 
it quantifies in how far it fails to be a unit vector field. 
Writing $E_h$ as a weighted average of squared finite difference
quotients (Lemma \ref{lem energy 1} \eqref{E finite diff})
shows its natural connection to the Dirichlet energy.
This can also be phrased in terms of the $\Gamma$-convergence of
$\frac12 E_h$ to the Dirichlet energy $\frac12 \int |\nabla u |^2 dx$.

The basis of our analysis is a minimizing movements interpretation of thresholding in our context of higher codimension. In resemblance to \eqref{MinMove},
given the state $u^{n-1}$ at the $(n-1)$-st step, the state $u^n$ is found 
by minimizing the functional
\Beqn\label{Intro MM}
\frac12 E_h(u) + \frac1{2h} \Norm{G_{\frac{h}{2}}*(u-u^{n-1})}^2_{L^2}
\Eeqn
(see Lemma \ref{lemma MM} below). 
The minimization principle \eqref{Intro MM} is in accordance with the gradient flow structure of the (harmonic map) heat flow, which is the $L^2$-gradient flow of the Dirichlet energy.

\medskip 
From \eqref{Intro MM}, we immediately
obtain some energy dissipation relation (Lemma \ref{lem estimate dtu}), which serves well as an a priori estimate.
However this estimate fails to fully capture the limiting dynamics as $h\to 0$.
To understand this well known fact, let us give some background on minimizing movements. 
While at first glance, a gradient flow $\partial_t u = - \nabla E(u)$ seems to need a smooth Riemannian structure, it is clear that the minimization problem \eqref{MinMove} makes sense in any metric space. 
This is the basis of De Giorgi's theory to define gradient  flows in metric spaces:
It is easy to see that the solution of a smooth gradient flow is characterized by the optimal rate of energy dissipation $\frac{d}{dt} E(u) \leq - \frac12 |\partial_t u|^2 -\frac12 |\nabla E(u)|^2 $ (since then equality holds in Young's and Cauchy's inequality).

It is worth noting that the energy dissipation rate $\frac{d}{dt} E(u) \leq -|\nabla E(u)|^2$ (or $\frac{d}{dt} E(u) \leq -|\partial_t u|^2$, respectively) is necessary but by no means does it characterize the solution.
Because of the degeneracy in the case of mean curvature flow however, it is more convenient to measure this rate only in terms of gradient of the area functional, i.e., the mean curvature. 
To then capture all information, one needs to monitor localized versions of the energy.
 Then under certain regularity conditions, this family of energy dissipation inequalities indeed characterizes the MCF.
 
When deriving either of these energy dissipation relations for limits of minimizing movements schemes, another technical difficulty appears: The a priori bound obtained from comparing $u^n$ to its predecessor in  \eqref{MinMove} fails to be sharp by a factor $\frac12$. However, this can be cured by considering the nonlinear interpolation between $u^{n-1}$ and $u^n$, choosing $u^{(n-1)+\lambda}$ to be a minimizer of $E(u)+\frac1{2\lambda h} d^2(u,u^{n-1})$. This interpolation was first proposed by De Giorgi and has been used recently for the localized energies of thresholding by Otto and the first author \cite{laux2017brakke}.

A nice feature of the current line of proof is that the technical difficulty if interpolating the state functions can be omitted as the precise prefactor of the metric term has no importance in our main estimate.
However, the estimate is still delicate in the sense that the prefactors of two other term need to match in order to cancel a diverging term.
This will be guaranteed by a new monotonicity formula.

\medskip

The foremost obstacle in the case of codimension two is the fact 
that the Dirichlet energy is not uniformly 
bounded near the filament which is exactly the place where the $u^n$ becomes 
singular. 
In fact $E_h(u^n)$ blows up with rate $\Abs{\log h}$ near the filament.
For remedy, we introduce a \emph{localized} version of $E_h$ 
(Definition \ref{threshold.energy} \eqref{def Eh local2}).
The localization is taken to be 
a truncated version of the the squared distance function 
$d^2(\cdot, \Gamma_t)$ to the actual solution of the filament MCF. 
The importance of distance function was first pointed out by
De Giorgi \cite{de1994barriers}. Later on, Ambrosio and Soner \cite{ambrosio1996level}
have used this idea to characterize higher codimensional geometric 
flows in terms of their distance function.
The work of Lin \cite{lin1998complex} further takes advantage of
the properties of the squared distance function to derive a localized energy 
dissipation law for the
the complex Ginzburg-Landau equation. (A more detailed description will be given 
at the beginning of Section \ref{sec:prop and lemmas}.)
Inspired by this last work,
exploiting the properties of the squared distance function, in particular
\eqref{phi heat eqn}--\eqref{phi expansion of heat operator}, 
we can similarly establish that the localized
thresholding energy is uniformly bounded (Proposition \ref{prop gronwall})
so that the location of the
singular set of $u^n$ exactly coincides with $\Gamma_t$.

Yet another new ingredient in the filament case is that 
we need to capture two equations:
(i)
the motion law of the filament---the set where $u$ is singular and 
(ii) the
evolution of the phase of $u$ \Em{away} from the filament.
The latter is due to the extra degree of freedom in the zero set $\mathbb{S}^1$
of $W$. To this end, we will derive two Euler-Lagrange equations 
(Lemma \ref{lem EL eq} \eqref{EL inner}--\eqref{EL outer weak})
for $u^n$ by considering \Em{inner} and \Em{outer} variations of $u^n$
\eqref{inner variations}--\eqref{outer variations}.
Note that for the thresholding scheme in the hypersurface case,
the limiting description is simply described by the MCF of the
interface $\Gamma$.
The state variable $u$ is identically equal to $+1$ and $-1$ away from
$\Gamma$.

Last but not the least, a sharp monotonicity formula (Lemma \ref{lem energy 2})
for the thresholding energy (as a function of the time step $h$) is used in a 
very crucial way. It is an improved version of an earlier 
``approximate monotonicity'' formula due to Esedo\u{g}lu-Otto in
\cite{EseOtt14}. So far it is only proved in codimension 
two and does not carry over immediately the the hypersurface 
case. In the hindsight, this is related to the fact that for complex-valued
$u$, in the Ginzburg-Landau functional \eqref{GLF}, for typical functions, 
the Dirichlet energy dominates the potential term. On the other hand,
for the scalar version (Allen-Cahn equation), there is equipartition of energy, i.e., the energy is equally distributed between the two terms. It will be interesting to investigate the monotonicity 
formula in a more general setting.

\subsection{Related work}\label{sec:related work}

Here we describe some related work on the analysis of thresholding scheme and 
some of its generalizations.

As mentioned before, the use of thresholding scheme for hypersurfaces has a 
long history, initiated by the work of Merriman, Bence, and Osher \cite{MBO92}. 
Soon after, rigorous proofs of convergence are given in 
Evans \cite{evans1993convergence} 
and Barles and Georgelin \cite{barles1995simple}. 
As these proofs are entirely based on the comparison principle, they are basically restricted to the special case of a single hypersurface.
Several recent works on thresholding scheme have overcome this restriction.
Esedo\u{g}lu and Otto's minimizing movements interpretation \cite{EseOtt14} generalized the scheme to arbitrary surface tensions and led to a series of conditional convergence results which are not based on the comparison principle but on the gradient-flow structure of MCF (of networks of interfaces).
Under the assumption that the total energy of the approximations converge to those of the limit, Otto and the first author proved that the limit solves a distributional formulation of motion by mean curvature, also in case of networks of hypersurfaces \cite{laux2015convergence}. (See also \cite{LauxSimon} for a 
proof of a similar result for a multi-phase Allen-Cahn system.)
Swartz and the first author extended these methods to incorporate external forces and volume constraints \cite{LauSwa15}.
Only recently, a \emph{local} minimality property of thresholding in the case of (networks of) hypersurfaces has been observed and used by Otto and the first author to prove that under the same assumption, this limit is also a unit-density Brakke flow \cite{laux2017brakke}. There the precise dissipation rate is fundamental in proving Brakke's inequality and the authors use De Giorgi's variational interpolations to obtain the precise constant.

The incorporation of anisotroy in the curvature motion is also of interest,
both mathematically and practically, due to again the simplicity of 
thresholding. One of the earliest work in this regard is 
\cite{ishii1999threshold}. 
It starts from a given (positive) convolution kernel and identity the
anisotropy. The ``inverse problem"---the construction of kernels for 
prescribed anisotropy---is considered in 
\cite{bonnetier2012consistency, elsey2016threshold}.
A technically difficult aspect is that in general the kernel must be 
necessarily non-positive. Hence the traditional proof of convergence is again not applicable in a straightforward way. Regarding this, another line of 
proof is constructed by Swartz and the second author \cite{SwaYip17}. 
This last work proves, by constructing an appropriate ansatz, consistency and
stability statements and a convergence rate of the scheme without using the 
comparison principle. Though currently it only considers the case of classical 
isotropic MCF, at the conceptual level, it can be applicable to a more general 
situation.

\medskip
Departing from the hypersurface case, in relation to the current paper,
we emphasize here work related to the motion of a filament in $\mathbb{R}^3$
which has codimension two. The convergence of the Ginzburg-Landau dynamics
to MCF when classical solution exists was proved in \cite{JerrardSonerMCF} and
\cite{lin1998complex}.
The work \cite{AmbrosioSonerVar} extended the result to varifold convergence
under the assumption that the density of the limit measure is bounded from 
below. This assumption was finally eliminated in the work 
\cite{bethuel2006convergence}.
A counter-part of the motion law \eqref{GL} is the consideration of 
Schr\"{o}dinger dynamics of \eqref{GLF}, written as
$i\partial_t u_\eps = \partial_u\Cal{F}_\eps(u)$. Heuristic
asymptotics lead to a limit singular measure that coincides with a filament 
evolving according to curvature motion along the \emph{bi-normal $\mathbf{B}$} 
to the curve. There are many
interesting open questions concerning this motion, regarding well-posedness
and approximation algorithms. We refer to \cite{jerrardICM} for a recent survey of this model.

Last but not the least, the very recent work
Osting and Wang \cite{osting2017generalized} discovered the same minimizing movements principle of the scheme \cite{ruuth2001diffusion} as ours
and used it to generalize the scheme from unit vector fields to matrix fields in $O(n)$, i.e., "unit" matrix fields w.r.t.\ the Frobenius norm. Furthermore, they provide promising numerical tests of this extension. While the limit $h\to0$ is not studied there, it seems that our convergence analysis, as described in Section \ref{sec:hmhf} should apply to this case as well. 
In the absence of singularities, the proof might simplify in the 
sense that one only needs to consider outer variations as in the case of 
Section \ref{sec:hmhf}. It might also be interesting to consider and classify the singularity structures that can appear.

\subsection{Structure of the paper}
In the next Section \ref{sec:filament}, we state our main convergence result 
\S\ref{sec:main res} and some remarks about it \S\ref{rem Eu0}.
Then in \S\ref{sec:prop and lemmas} we list all the technical lemmas
to be used. We highlight here the
(localized) minimizing movements principles Lemma \ref{lemma MM},
(localized) energy dissipation law Proposition \ref{prop gronwall},
the sharp monotonicity lemma \ref{lem energy 2}, and
the two Euler-Lagrange equations for the minimizers during each time
step Lemma \ref{lem EL eq}.
These are all proved in Section \ref{sec:proofs} which forms the bulk of 
the paper.
In Section \ref{sec:pinning}, we discuss further insights from the variational 
viewpoint, namely boundary conditions, vortex motion in two dimensions and the 
less singular harmonic map heat flow in higher dimensions.

    \section{Mean curvature flow of a filament in \texorpdfstring{$\R^3$}{R3}}\label{sec:filament}

Throughout the paper we will assume that a smooth mean curvature 
flow of an embedded curve $\Gamma_t$ starting from $\Gamma_0$ exists 
up to some time $T > 0$.
The flow can be expressed in terms of a parametrization 
$\gamma(\cdot, t)$ such that
$\Gamma_t=\gamma([0,1],t)$, where
\Beqn\label{MCF0}
\gamma: [0,1]\times[0,T)\MapTo\R^3,
\Eeqn
and it satisfies
     \begin{equation}\label{MCF}
      \frac{\partial}{\partial t}\gamma(\theta,t)
=  \kappa_{\gamma(\theta,t)}\mathbf{N}_{\gamma(\theta,t)},
\,\,\,\,\,\,\gamma([0,1], 0) = \Gamma_0.
     \end{equation}
The boundary conditions at $\theta=0$ and $\theta=1$ will
be specified later.
In the above, $\kappa$ denotes the curvature of $\Gamma_t$
at $\theta\in [0,1]$ and $\mathbf{N}$ is
unit normal vector pointing in direction of the derivative of the 
unit tangent vector. Although $\mathbf{N}$ is not defined when $\kappa =0$, 
the product $\kappa \mathbf{N}$ is always well-defined.
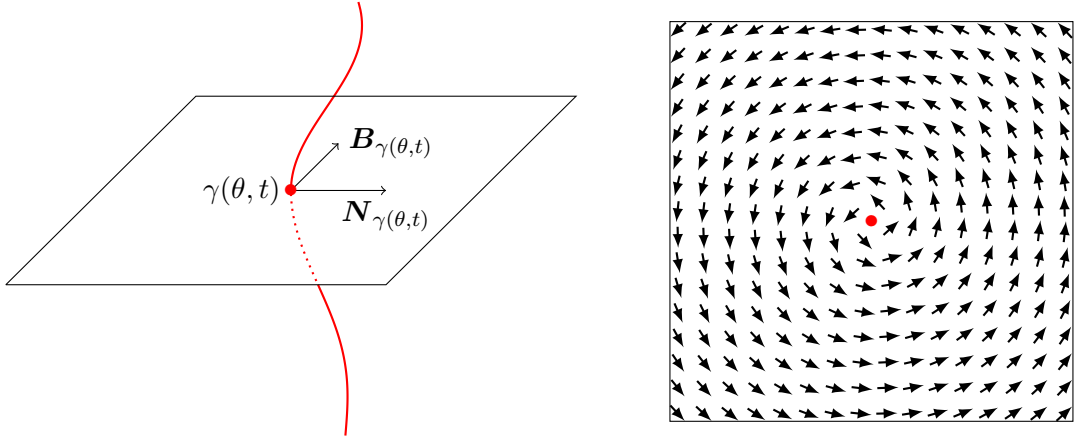
\begin{figure}
\centering
\begin{tikzpicture}[scale=2.5]
\draw [domain=-.48:0,smooth,variable=\x,red,dotted,thick] plot ({{-(cos(200*\x)-1)/10+\x*\x/10+3/2},{\x+1/2}});
\draw [domain=-1.3:-.5,smooth,variable=\x,red,thick] plot ({{-(cos(200*\x)-1)/10+\x*\x/10+3/2},{\x+1/2}});
\draw [domain=0:1,smooth,variable=\x,red, thick] plot ({{-(cos(200*\x)-1)/5-\x*\x/30+3/2},{\x+1/2}});
\draw  (0,0) -- (2,0) -- (3,1) -- (1,1) -- (0,0);
\draw [->] (3/2,1/2) -- + (1/2,0) node [below] {$\boldsymbol{N}_{\gamma(\theta,t)}$};
\draw [->] (3/2,1/2) -- + (1/4,1/4) node [right] {$\boldsymbol{B}_{\gamma(\theta,t)}$};

\node [red] at (3/2,1/2) {$\bullet$};
\node [black,left] at (3/2,1/2) {$\gamma(\theta,t)$};

\end{tikzpicture}
$\qquad$
\begin{tikzpicture}[scale=1]
 \draw (0.2,0.2) rectangle (5.5,5.5);  
  \begin{axis}
[
      view     = {2.3}{90},
      domain   = -1:1,
      samples  = 16,
      axis lines=none,
      axis equal
    ]
    \addplot3 [ thick, quiver={u={-y/(x^2+y^2)^(.5)}, v={x/(x^2+y^2)^(.5)}, scale arrows=.11,
               every arrow/.append style={-latex}}] (x,y,0);
  \end{axis}
   \node [red] at (2.85,2.85) {$\bullet$};
\end{tikzpicture}
\caption{Left: A filament in $\R^3$ with a horizontal slice of $\R^3$. 
Right: The initial conditions $u^0$ on the depicted slice ``wind" around the filament.}
\label{fig:slices}
\end{figure}
Note that short-time existence has been established in a very general framework by Gage and Hamilton \cite[Section 2]{gage1986heat}. 
See also Huisken-Polden \cite{HuiskenPoldenGraphProof} for solving the equation
using a graph coordinate system and \cite{WangMuTaoReview1, WangMuTaoReview2} 
for reviews of higher codimensional mean curvature flows.

We will state the algorithm of the thresholding scheme in terms
of the heat kernel on $\R^d$,
\Beqn
G_h(z)=\frac1{(4\pi h)^{d/2}} \exp\Big(-\frac{|z|^2}{4h}\Big)
\Eeqn
which is the solution operator of the linear heat equation and it solves
\begin{equation}\label{G solves heat eqn}
\partial_h G - \Delta G =0,
\quad G_{0} =\delta_0.
\end{equation}
Some basic facts about $G_h$ will be collected 
at the beginning of Section \ref{sec:proofs}.
    
    \medskip
    
We now present the algorithm for the filament thresholding scheme and 
our main convergence result. In order not to be distracted by
boundary conditions, we will work with periodic boundary condition.
Recall that in this setting, the configuration
or ``state'' of the algorithm at each time step is given by an 
$\R^2$- or $\mathbb{C}$-valued function $u$ defined on $\domain$. 
To be more precise, after each projection step, $u$ has unit length,
i.e., it is $\mathbb{S}^1$-valued.

    \subsection{Main result}\label{sec:main res}
    \begin{algo}\label{algo filament}
      Given a time-step size $h>0$ and the configuration 
$u^{n-1}:\domain\to\mathbb{S}^1$ at time $t=(n-1)h$, construct the configuration $u^n$ at time $t=nh$ by the following two operations:
      \begin{enumerate}
       \item Diffusion: convolve $u^{n-1}$ with the heat kernel, i.e.,
set $v^n:= G_h\ast u^{n-1}$;
       \item Projection: project $v^n$ onto 
the unit sphere, i.e., set $u^n := \frac{v^n}{|v^n|}$.
      \end{enumerate}
    \end{algo}
    We denote by $u^h$ the piecewise constant interpolation in time of the 
functions $u^0,u^1,\dots$ defined by
    \Beqn\label{def pw-const}
     u^h(x,t) := u^n(x) \quad \text{for } t\in[nh,(n+1)h).
    \Eeqn
We also define the following backward in time finite difference quotient
for any time dependent function $v$,
\Beqn\label{bw time diff}
\partial_t^h v := \frac{v(t) - v(t-h)}{h}.
\Eeqn

     Our main result is the following convergence of $u^h$
to a filament moving by its mean curvature.
    \begin{thm}\label{thm filament}
      Let $\Gamma_t$, $t\in[0,T)$ be a filament evolving smoothly by 
mean curvature flow \eqref{MCF0}-\eqref{MCF} 
in $\domain$ and assume that the initial conditions for Algorithm \ref{algo filament} are well-prepared in the sense of Definition \ref{def initial}.  Then the approximate solutions $u^h$ \eqref{def pw-const} 
obtained by Algorithm \ref{algo filament} 
converge as $h\downarrow 0$ to $\Gamma_t$ in the following sense.
      
      For every sequence $u^h$, there exists a subsequence 
      (still denoted by $u^h$) and a vector field 
$u \in H^1_{\text{loc}}((\domain\times(0,T)) \setminus \Gamma;\mathbb{S}^1)$ 
such that
      \begin{align}
      u^h \to &u & &\text{in } L^2(\domain \times (0,T))\label{comp u},\\
     \nabla G_{h} \ast u^h \rightharpoonup& \nabla u&  &\text{in }L^2_{\text{loc}}((\domain\times(0,T) ) \setminus \Gamma),\label{comp Du}\quad \text{and}\\
     \partial_t^h\big(G_{h} \ast u^h\big) \rightharpoonup &\partial_t u & &\text{in }L^2_{\text{loc}}((\domain\times(0,T) ) \setminus \Gamma)\label{comp dtu}.
    \end{align}
      In the limit $h\to 0$, the vorticity set concentrates only on 
$\Gamma_t$ in the sense that the Dirichlet energy of $u$ 
      stays bounded away from $\Gamma_t$.
      Furthermore, $u$ solves the harmonic map heat flow equation away 
from $\Gamma$.
    \end{thm}

We pause to elaborate the above statement. Further explanation and remarks 
will be given in Section \ref{rem Eu0}.

Our paper crucially makes use of the following energy functional and its
localized version.
\begin{defn}[Thresholding energy]\label{threshold.energy}
Let $h>0$.
For any unit vector field, $u:\domain\to\mathbb{S}^1$, we define the energies
    \Beqn\label{def Eh local}
     E_h(u) := \frac1h \int \left(1-u\cdot G_h\ast u\right) dx.
    \Eeqn
and its localized version, which is 
defined for any $\psi :\domain\to\mathbb{R}$ as:
    \Beqn\label{def Eh local2}
     E_h(u, \psi) := \frac1h \int \psi\left(1-u\cdot G_h\ast u\right) dx.
    \Eeqn
\end{defn}
As to be seen later in Lemma \ref{lem energy 1}, the above functionals 
approximate the Dirichlet energy $\int |\nabla u|^2dx$ of $u$.

 \medskip
  
    The following lemma is the basis of our analysis. It states that---similar to thresholding for hypersurfaces, cf.\ \cite{EseOtt14},---also in our case of higher codimension, thresholding can be interpreted as a minimizing movements scheme.
    Furthermore, the lemma establishes a localized version of this minimizing movements interpretation similar to the one for hypersurfaces in \cite{laux2017brakke}.

    \begin{lem}\label{lemma MM} 
    Each time step $u^{n-1} \mapsto u^{n} = \frac{G_h\ast u^{n-1}}{|G_h\ast u^{n-1}|}$ of the thresholding scheme (Algorithm \ref{algo filament}) 
    is equivalent to minimizing
    \begin{equation}\label{MM}
     E_h(u) + \frac1h \int \left( u-u^{n-1}\right) \cdot  G_h\ast\left( u-u^{n-1}\right) dx
    \end{equation}
    among all $u\colon \domain \to \R^2$ with $|u|\leq 1$ a.e.
    In particular, we have the following energy-dissipation estimate for the piecewise constant interpolation $u^h$ \eqref{def pw-const},
    \begin{equation}\label{energy estimate}
     E_h(u^h(T)) + \int_0^T \int \big| G_{h/2}\ast \partial_t^h u^h\big|^2 \,dx\,dt \leq E_h(u^0).
    \end{equation}
    Furthermore, for any non-negative test function $\psi\geq0$, $u^{n}$ 
    minimizes the following localized version of \eqref{MM},
    \begin{equation}\label{MM local}
     E_h(u,\psi) + \frac1h \int \psi \left( u-u^{n-1}\right) \cdot  G_h\ast\left( u-u^{n-1}\right) dx
     +\frac1h \int \left(u - u^{n-1}\right) \cdot \left[ G_h \ast, \psi \right] u^{n-1} \,dx
    \end{equation}
    among all $u\colon \domain \to \R^2$ with $|u|\leq 1$ a.e.
    \end{lem}
The analogy of the energy dissipation law \eqref{energy estimate}
for \eqref{MM local} with appropriate choice of the localization function 
$\psi$ is the key technical result of our approach
and will be stated in Proposition \ref{prop gronwall}.

In the formula \eqref{MM local}, note the appearance of the
commutator $[G_h\ast, \psi]$ between the 
convolution with $G_h$ and multiplication by $\psi$ which is defined as:
\Beqn\label{commutator}
[G_h\ast, \psi]f = G_h\ast(\psi f) - \psi(G_h\ast f).
\Eeqn
    
\begin{defn}\label{def initial}
    		The initial datum $u^0 \colon \domain \setminus \Gamma_0 \to \mathbb{S}^1$ is called well-prepared if the following two conditions hold:
    		\begin{enumerate}
    		\item the approximate energies blow up logarithmically: there exist constants $0<c<C<\infty$ such that
      \begin{equation}\label{Eu0}
      c\left|\log h\right|  \leq E_h(u^0) \leq C \left|\log h\right|;
      \end{equation}
       \item away from the filament $\Gamma^0$, the approximate energies stay bounded
      \begin{equation}\label{Eu0phi}
       E_h(u^0,\phi_\sigma(0)) \leq C(\sigma),
      \end{equation}
      where $\phi_\sigma(0)$ denotes the smoothly truncated squared distance function to $\Gamma_0$ defined in \eqref{def phi}. 
Here $\sigma$ is some positive number depending only on the
curves $\Gamma_t, t\in[0,T]$ but not on $h$.
      \end{enumerate}
    \end{defn}

\subsection{Remarks about the main result}\label{rem Eu0}
Here we give some further remarks about our main convergence result.

(1) Similar to the Ginzburg-Landau approximation in the case of vortex in $\mathbb{R}^2$, for smooth initial conditions $\Gamma_0$ close to a straight 
line parallel to the $x_3-$axis, one can easily construct initial data $u^0$ satisfying
Definition \ref{def initial}.
Specifically, let $\Gamma_0 \subset \R^3$ be a curve given in the 
form of 
\[
\Gamma_0 = \{(\gamma(x_3),x_3) \colon x_3\in[0,1)\},
\]
where $\gamma\colon [0,1)\to \R^2$ is a smooth $1$-periodic 
function. Then
\Beqn\label{construct u0}
u^0(x) := \pm \frac{(x'-\gamma(x_3))^\perp}{|x'-\gamma(x_3)|}\quad \text{for }x=(x',x_3)\in \R^3\setminus \Gamma_0
\Eeqn
is well-prepared (locally around the filament $\Gamma_0$). 
The precise computation is shown in Appendix \ref{appendix:u0}.
Another equivalent choice is to simply use the radial vector in the normal plane of the curve, i.e., \eqref{construct u0} without the rotation  $\perp$.

(2) The bulk of the paper is presented in the easiest case of periodic boundary conditions. We only assume this to omit technical difficulties which would pollute the proof. The interested reader is referred to the discussions in Section \ref{sec:pinning} considering boundary conditions and Appendix \ref{appendix whole space} considering the whole space.

Our proof does not use the fact that we work in three dimensions. In fact the proof applies word by word for any codimension-two mean curvature flow. Then the vector field $u$ is defined on (subsets of) $\R^d$ with values in $\mathbb{S}^1$ and the energy $E_h$ takes the exact same form.
However, we prefer to keep the language simple and restrict ourselves to the physically most relevant case of a filament in $\R^3$.

(3) The term vorticity set refers to the 
support of the limit of the measures
corresponding to the rescaled thresholding energies.
To be precise, we define the measure $\mu_h$ as
\[
\mu_h(t) = 
\frac1{\Abs{\log h}h} \Brac{1-u^h(x,t)\cdot G_h* u^h(x,t)}dx.
\]
Then the vorticity set is given by $\supp (\mu(t))$ where 
$\mu(t) = \lim_h \mu_h(t)$.
The analogous concept also exists for the Ginzburg-Landau dynamics \eqref{GL}
in which the measure is defined as
\[
\mu_\eps(t) = \Lover{\Abs{\log\eps}}\Brac{
\Lover{2}\Abs{\nabla u_\eps(x,t)}^2 + \Lover{4\eps^2}(1-\Abs{u_\eps(x,t)}^2)^2}\,dx
\]
For the case of filament motion, Lin \cite[Theorem 4.1]{lin1998complex} 
showed that the limit $\mu(t) = \lim_{\eps\to0}\mu_\eps(t)$ satisfies
$c_1\mathcal{H}^1\lfloor\Gamma_t
\leq \mu(t) \leq c_2\mathcal{H}^1\lfloor\Gamma_t$ for some constants $c_1$
and $c_2$.
This is consistent with the fact that the limit of $u_\eps(\cdot, t)$
still winds around $\Gamma_t$ with winding number one. Our current result
only states that the limit $\mu(t)=\lim_{h\to0} \mu_h(t)$ satisfies
$\mu(t) \leq c_2\mathcal{H}^1\lfloor\Gamma_t$. However, we expect a lower
bound to be feasible if one can show that the thresholding energy is
bounded from below by $c\abs{\log h}$ for typical functions $u$ with a nontrivial winding number around a curve. We leave this latter 
statement to a future project.

(4) Note that the limit description is given by two dynamical equations. 
One is for the vorticity set $\Gamma_t$ which evolves by MCF. The other is the
evolution of $u$ away from $\Gamma_t$. It is given by the harmonic
heat flow on $\mathbb{S}^1$. Precisely, $u$ satisfies
\Beqn\label{HHF}
\partial_t u = \Delta u + \Abs{\nabla u}^2 u
\Eeqn
which is the ($L^2$-)gradient flow for the Dirichlet energy
\MathSty{\frac12\int\Abs{\nabla u}^2\,dx} for $\mathbb{S}^1$-valued function $u$.
Another maybe more transparent description can also be given. Away from
$\Gamma_t$, if we write $u(x,t)$ (locally) as $e^{i\theta(x,t)}$ for some
phase function $\theta$, then \eqref{HHF} is equivalent to
\[
\partial_t \theta = \Delta \theta.
\]

(5) The result stated in \eqref{comp Du} 
gives that $u^h$ converges weakly to $u$ in $H^1$. We believe this can be 
improved to be strong convergence in $H^1$. Such a statement is proved for the 
Ginzburg-Landau dynamics \eqref{GL} \cite[Section 5]{lin1998complex}. 
The usual strategy in establishing this is to show that small energy implies 
that $\abs{u}$ is close to one, and then higher order regularity is proved
by means of some blow-up argument. It will be interesting to have similar
statement for the thresholding scheme. Furthermore, it is also of
practical importance to have a convergence rate. We defer these issues to 
future works.

\medskip

    Throughout the paper we will make use of the following notation.
    By $C$ and $C(\sigma)$ we denote generic constants independent of the time-step size $h$, where $C(\sigma)$ may depend on the parameter $\sigma$.
The dependence on $\sigma$ is not important in this paper.
In particular we may allow $C(\sigma)$ tend to zero or infinity as 
$\sigma\to0$. However, the asymptotics in terms of $h$ is crucial in our 
analysis and will be spelled out explicitly.
We write $A\lesssim B$ if there exists a generic constant $C<\infty$ such that $A\leq C\,B$. 
If a quantity $A$ stays bounded by $B$ as $h\to0$, we write $A=O(B)$.
The same applies for $A=o(B)$, which means $\frac AB \to0$ as $h\to0$. 
In particular, we will use $O(1)$ and $o(1)$ referring to constants
which are bounded and convergent to $0$ as $h\to 0$.
For simplicity, we often omit the notation $h\to0$.
Furthermore, to describe asymptotics at the heuristic level, we 
often use the symbol $\approx$ which will always be followed by rigorous 
explanations.
By $\int \,dx $ we denote the integral $\int_{\domain} \,dx$, while $\int \,dz$ denotes the integral $\int_{\R^3} \,dz$.    
    

    \subsection{Main propositions and lemmas}\label{sec:prop and lemmas}
    
    We assume the existence of a smooth mean curvature flow 
$\Gamma_t$ \eqref{MCF0}-\eqref{MCF} and will exploit the
properties of the squared distance function to $\Gamma_t$ 
to construct a localization function.
Precisely, for $\sigma > 0$, we consider the function
$\phi=\phi_\sigma(x,t)$
which at any time $t$ is (a truncated version of) the squared distance to the 
curve $\Gamma_t$ defined as
    \begin{equation}\label{def phi}
      \phi(x,t) = \phi_\sigma(x,t) :=  \frac12 f_\sigma(d(x,\Gamma_t)),
    \end{equation}
    where $d(x, \Gamma) := \inf \{ |x-y| \colon y\in \Gamma\}$ is the distance function to the set $\Gamma$ and $f_\sigma\colon (0,\infty)\to (0,\infty)$ is a smooth monotone non-decreasing function such that
    \begin{equation}\label{def f}
      f_\sigma(\rho) =\begin{cases}
                       \rho^2 &\text{if } \rho <\sigma\\
		       4\sigma^2& \text{if } \rho>2\sigma,
                      \end{cases}
    \end{equation}
    cf.\ Fig.\ \ref{fig:cutoff}.
    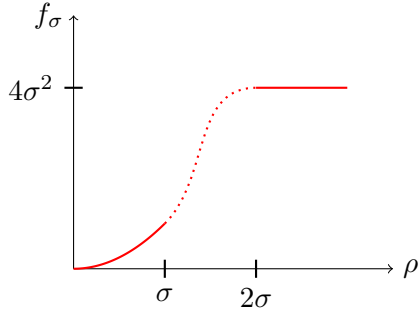
\begin{figure}
    \centering
      \begin{tikzpicture}[scale=1.2]
	\draw [->] (0,0) --  (3.5,0) node[right]{$\rho$};
	\draw [->] (0,0) --  (0,2.8) node[left]{$f_\sigma$};
	\draw[thick,domain=0:1,smooth,variable=\radius,red] plot ({\radius,{\radius*\radius/2}});
	\draw[thick,domain=2:3,smooth,variable=\radius,red] plot ({\radius,{2}});
	\draw[thick,domain=4:6,smooth,dotted,variable=\radius,red] 
	(1,.5) to[out=45,in=180] (2,2);
	\draw [thick](1,.1) -- (1,-.1) node[below] {$\sigma$};
	\draw [thick](2,.1) -- (2,-.1) node[below] {$2\sigma$};
	\draw [thick](.1,2) -- (-.1,2) node[left] {$4\sigma^2$};
      \end{tikzpicture}
      \caption{The smooth cut-off of the profile $\rho\mapsto\rho^2$.}
       \label{fig:cutoff}
    \end{figure}
    By \cite[Lemma 3.7]{ambrosio1996level}, the gradient of $ \phi$ solves the heat equation 
    \begin{equation}\label{phi heat eqn}
       \partial_t \nabla\phi -\Delta \nabla \phi  =0 \quad \text{on }\Gamma_t.
    \end{equation}
    The eigenvalues of the Hessian $\nabla^2 \phi$ are 
well-controlled \cite[Theorem 3.2]{ambrosio1996level}. 
In particular, near $\Gamma_t$, the Hessian has two eigenvalues equal to
one and the third strictly less than one. (For a general codimension-2 surface 
in $\R^d$, the Hessian has two eigenvalues equal to one and the remaining $d-2$ eigenvalues are strictly less than one.)
From these, we deduce that for small $\sigma>0$, 
the following holds:
    \begin{equation}\label{phi hessian}
\nabla^2 \phi \leq Id \quad \text{on } \{(x,t) \colon d(x,\Gamma_t) <\sigma\},
\,\,\,\text{i.e.,}\,\,\,
\text{$\xi \cdot \nabla^2 \phi \,\xi \leq  |\xi|^2$ 
for any $\xi\in \R^3$.}
    \end{equation}
Furthermore, we have
    \begin{equation}\label{phi derivatives on gamma}
     \partial_t\phi= 0 ,\quad \text{and}\quad \Delta \phi = 2\quad \text{on } \Gamma_t.
    \end{equation}
    Now applying Taylor expansion to $\partial_t \phi-\Delta \phi$, we obtain by \eqref{phi derivatives on gamma} and \eqref{phi heat eqn} that 
    \begin{equation}\label{phi expansion of heat operator}
      \partial_t \phi - \Delta \phi \leq - 2 + \frac {C(\sigma)}2 d^2(x,\Gamma_t) = -2 + C(\sigma) \phi \quad \text{in } \{(x,t) \colon d(x,\Gamma_t) <\sigma\}.
    \end{equation}
We point out here that the use of the squared distance function is not 
coincidental as it was used to characterize mean curvature flow in the sense 
that $\Gamma_t$ evolves by mean curvature flow 
\Em{if and only if} $\phi$ solves
\eqref{phi heat eqn} \cite[Lemma 3.7]{ambrosio1996level}.

We are very much inspired by the following localized version of the 
energy-dissipation relation \eqref{GL ED} for the Ginzburg-Landau dynamics
\eqref{GL} derived in \cite[p. 421-422]{lin1998complex}.
A direct computation followed by an application of \eqref{phi derivatives on gamma} and \eqref{phi expansion of heat operator} shows
\begin{equation}\label{GL ED loc}
\begin{split}
\frac{d}{dt} \int &\phi_\sigma  \left(  \frac12 \left| \nabla u_\eps\right|^2 +\frac1\eps W(u_\eps)\right) dx\\
&= - \int \phi_\sigma \left|\partial_t u_\eps \right|^2 +\sum_{i,j}\partial_i \partial_j \phi \,\partial_i u \,\partial_j u + \left( \partial_t -\Delta \phi\right) \left(  \frac12 \left| \nabla u_\eps\right|^2 +\frac1\eps W(u_\eps)\right) dx\\
&\leq  -\int  \phi_\sigma \left|\partial_t u_\eps \right|^2dx+C(\sigma) \int \phi_\sigma \left(  \frac12 \left| \nabla u_\eps\right|^2 +\frac1\eps W(u_\eps)\right) dx.
\end{split}
\end{equation}
Then a Gronwall argument gives that the (localized) Ginzburg-Landau
energy stays bounded away from the filament.
Note that in the above, equations \eqref{phi hessian} and 
\eqref{phi expansion of heat operator} are used in a subtle but crucial
way,
\[
\begin{split}
& \sum_{i,j}\partial_i \partial_j \phi \,\partial_i u \,\partial_j u 
+ \left( \partial_t -\Delta \phi\right) 
\left(  \frac12 \left| \nabla u_\eps\right|^2 \right)\\
\leq &
\Abs{\nabla u_\eps}^2 + (-2 + C(\sigma)\phi)\left(  \frac12 \left| \nabla u_\eps\right|^2 \right)
\leq C(\sigma) \frac{\phi}{2}\Abs{\nabla u_\eps}^2.
\end{split}
\]
Curiously, our monotonicity formula (Lemma \ref{lem energy 2}) is in a 
sense an analogue of this property and will also be used in a crucial step.

     To mimic the previous computation, the localized thresholding 
energies \eqref{def Eh local2} play a pivotal role in our analysis.
The following proposition is the key ingredient in our proof of Theorem \ref{thm filament}.
Essentially, it provides a localized energy-dissipation inequality very much 
like the time integrated version of \eqref{GL ED loc}.

    \begin{prop}[Energy inequality]\label{prop gronwall}
Let $\Gamma_t$ (for $0\leq t \leq T$),
$\phi$, and $u^h$ be given in 
\eqref{MCF0}-\eqref{MCF}, 
\eqref{def phi}, 
and  \eqref{def pw-const}.
Then as $h\to0$, we have
      \begin{equation}\label{gronwall}
	  E_h(u^h(T),\phi_\sigma (T)) + \int_0^T \int \phi_\sigma \, \big| G_{h/2}\ast \partial_t^h u^h \big|^2 dx\,dt \leq C(\sigma) E_h(u^0,\phi_\sigma(0)) +o(1).
	\end{equation}
    \end{prop}

Its proof relies on several results which we present next.

\medskip

     The following basic facts about the energy $E_h$
are stated for more general localization function.
     \begin{lem}\label{lem energy 1}
     Let $u\colon \domain \to \mathbb{S}^1$ be a unit vector field
and $\psi\colon\domain \to \R$ a bounded function. 
     \begin{enumerate}[(i)]
      \item 
The approximate energy $E_h$ can be written as the following weighted integral of finite differences:
      \begin{align}\label{E finite diff}
	E_h(u,\psi) = \frac12  \int \psi(x) \int  G(z)\, 
\left| \frac{u(x)-u(x-\h z)}{\h}\right|^2 dz\, dx.
      \end{align}
      Furthermore, the energies $E_h$ approximate the Dirichlet energy in the sense that
      \Beqn\label{approx Dir}
       \lim_{h\downarrow0}E_h(u,\psi) =
\int \psi \left|\nabla u\right|^2 dx \quad \text{
      for unit vector fields } u\in W^{1,2}.
      \Eeqn
      \item If $\psi \geq 0$,
then the energy satisfies the following approximate monotonicity formula,
      \begin{equation}\label{monotonicity}
       E_{N^2 h}(u,\psi) \leq E_h(u,\psi) + C\Norm{\nabla\psi}_\infty
       \sqrt{N^2 h} \,E_h(u),
 \quad \text{for $N\in \N$.}
      \end{equation}
      \item If $\psi \geq 0$, then
the Dirichlet energy of $u$ is controlled by the energy:
      \begin{equation}\label{D and E easy}
        \int \psi \left| \nabla G_{h/2}\ast u\right|^2 dx \lesssim E_h(u,\psi).
      \end{equation}
    \end{enumerate}
    \end{lem}
    
We note the additional property that in fact the energies $E_h$ 
$\Gamma$-converge 
to the Dirichlet energy. This is because by (ii), the the pointwise 
convergence in (i) is almost monotone.
    
    The following lemma sharpens the monotonicity statement of 
   \eqref{monotonicity}
    in the case of $\phi \equiv 1$. It has an interesting implication, 
namely a ``sharp'' version of the comparison between 
the approximate energy $E_h$ of $u$ and
    the Dirichlet energy of its convolution $G_{h/2}\ast u$ in Lemma \ref{lem energy 1} (iii). As mentioned earlier, this sharp inequality will play a crucial 
role in our analysis.  
    
    \begin{lem}[Monotonicity]\label{lem energy 2}
      The approximate energies $E_h$ are monotone in $h$, i.e., for any fixed measurable $u\colon \domain \to \mathbb{S}^1$, we have
      \begin{equation}\label{dhE}
       \frac{d}{dh} E_h(u) \leq 0.
      \end{equation}
      Furthermore, we have the sharp inequality
      \begin{equation}\label{D and E sharp}
	\int |\nabla G_{h/2} \ast u|^2 \,dx \leq E_h(u).
      \end{equation}

    \end{lem}

    The following lemma gives a bound $O(\Lover{h}|\log h|)$ for the squared $L^2$-norm of the 
    discrete time derivative of the approximate solutions.
    Note that the bound diverges as $h\converge 0$.
    We believe this bound is far from optimal but 
it is sufficient for our purposes. 
    \begin{lem}\label{lem estimate dtu}
     Let $u^h$ be defined in \eqref{def pw-const}.
Then it satisfies the following a priori estimate.
     \begin{equation}\label{dtu}
	\int_0^T\int \big|\partial_t^h u^h\big|^2\,dx\,dt \lesssim \left(1+ \frac Th\right) E_h(u^0).
     \end{equation}
    \end{lem}
On the other hand, using the energy dissipation \eqref{energy estimate}, 
we automatically have the following ``better" estimate if the backward 
in time finite difference is smoothed out by convolving on the length scale $\h$:
     \begin{equation}\label{G_hdtu}
	\int_0^T\int \big|G_{h/2}*\partial_t^h u^h\big|^2\,dx\,dt 
\leq E_h(u^0).
     \end{equation}
Both \eqref{dtu} and \eqref{G_hdtu} will be used in our proof.

\medskip

    Recall that we will recover \emph{two} equations to describe the limit
of $u^h$ as $h\Converge 0$. 
The first is the motion law of the vorticity set 
which in the limit is a curve moving by its curvature. 
The second is the equation for the phase function 
which lives away from the vorticity set and in the limit 
solves a diffusion equation.
    To this end, we consider \Em{inner} and \Em{outer variations} of $u^n$
leading to \emph{two} of Euler-Lagrange equations for the minimization problem 
\eqref{MM}.
\begin{itemize}
\item
    The inner variation $u_s$ of $u\colon \domain \to S^1$ is
given by the variation of domain along a smooth vector field $\xi$:
    \begin{equation}\label{inner variations}
u_s(x) = u\big(x - s\xi(x)\big),
   \quad
    \text{so that}
   \quad
     \partial_s u_s\big|_{s=0} = -\xi \cdot \nabla u.
    \end{equation}
\item
    The outer variation $\tilde u_s$ of $u$ in direction of a smooth vector field $\varphi$ is given by
    \begin{equation}\label{outer variations}
      \tilde u_s := \frac{u+s\varphi}{|u+s\varphi|},
   \quad
    \text{so that}
   \quad
     \partial_s \tilde u_s\big|_{s=0} = \left( Id -u \otimes u \right) \varphi.
    \end{equation}
\end{itemize}
    
    Using the above, we have the following statements.
    \begin{lem}[Euler-Lagrange equations]\label{lem EL eq}
Let $u^h$ be the piecewise constant in time interpolation \eqref{def pw-const}.
      Then it satisfies the following two statements.
      \begin{enumerate}[(i)]
       \item  For any smooth vector field $\xi\colon \domain \times [0,T]\to \R^3$ we have
      \begin{align}
       & 2 \iint G_{h/2} \ast \partial_t^{h}u^h \cdot \left(\xi \cdot \nabla \right) G_{h/2} \ast u^h \,dx\,dt \label{EL inner} \\
&=\frac1h \iint \!\left(\nabla \cdot \xi\right) ( 1-u^h \cdot G_h\ast u^h) \,dx\,dt
       -2 \sum_{i,j}\iint \!\partial_i \xi_j \,\partial_i G_{h/2}\ast u^h \cdot \partial_j G_{h/2} \ast u^h \,dx\,dt +o(1).
\notag 
 \end{align}
      \item    For any smooth function $\zeta\colon \domain \times [0,T] \to \R$, we have
      \begin{equation}\label{EL outer weak}
        \iint (u^h_j\partial_t^{-h} G_h \ast u^h_i  - u^h_i \partial_t^{-h} G_h \ast u^h_j) \zeta +( u^h_j\nabla G_h \ast u^h_i - u^h_i\nabla G_h \ast u^h_j)\cdot \nabla \zeta \,dx\,dt=o(1).
      \end{equation}
      \end{enumerate}
    \end{lem}

Now we get to the proofs for all the statements and the main result.

    \section{Proofs of the lemmas and the main result}\label{sec:proofs}

We first state some basic facts about the heat kernel 
which will be used frequently.
Recall the notation for the heat kernel on $\R^d$:
\[
G_h(z) :=
\frac1{(4\pi h)^{d/2}} \exp\Big(-\frac{|z|^2}{4h}\Big)
\quad z\in \R^d,\quad h> 0.
\]
The following semi-group and factorization properties hold for $G$:
\Beqn\label{semigroup-factor}
\begin{array}{ll}
G_{s+t} = G_s \ast G_{t} & \text{for $s,t> 0$},\\
G_h(z) = G^1_h(z_1)\, G^{d-1}_h(z') & \text{for } z= (z_1,z'),
\quad z_1\in\R,\quad z'\in\R^{d-1}
\end{array}
\Eeqn
where $G^1$ and $G^{d-1}$ are the one- and $(d-1)$-dimensional Gaussian 
kernels respectively.
We also have the following statements about $G_h$:
\Beqn
0 \leq G_h(z) \lesssim \frac1{h^{d/2}}, 
\quad \int_{\R^d} G_h(z)\,dz=1, 
\quad \int_{\R^d} {\frac{|z|^2}{h}}G_h(z)\,dz=2, 
\Eeqn
\Beqn
\nabla G_h(z) = - \frac z{2h} G_h(z),
\quad |\nabla G_h(z)| \lesssim \frac1{\h} G_{2h}(z),
\Eeqn
\Beqn
\nabla^2 G_h(z) 
=  \left({ \frac{z}{2h} \otimes \frac{z}{2h}}- \frac{1}{2h}{Id} \right)G_h(z).
\Eeqn
Due to the symmetry of the heat kernel, $G_h(x-z) = G_h(z-x)$, the
convolution with $G_h$ is self-adjoint in the $L^2$-sense:
\Beqn\label{selfadj}
\InnProd{f}{G_h*g}_{L^2} =
\int f(x) (G_h* g)(x)\,dx = 
\int (G_h *f)(x) g(x)\,dx = 
\InnProd{G_h *f}{g}_{L^2}.
\Eeqn

Finally, we have the following expansion of the commutator \eqref{commutator} 
between $G_h*$ and multiplication by a 
test function $\psi$. For later convenience, it is stated for $G_{h/2}$.
\begin{lem}\label{lem:commutator.est}
For $\psi: \domain\MapTo\R$ and $V:\domain\MapTo\R^d$, it holds that
\begin{align}
       \Brac{\frac1h \left[ G_{h/2} \ast, \psi \right] V}(x)
       = &\nabla \psi(x) \cdot \Brac{\nabla G_{h/2} \ast V}(x)
       + O\big(\|\nabla^2 \psi\|_\infty \frac{|z|^2}{h} G_{h/2} \ast |V|\big).\label{comm V}
\end{align}
\end{lem}
It will be seen that the first term of \eqref{comm V} dominates the
second. Hence we will often write the commutator asymptotically as
\Beqn\label{comm asym}
\frac1h \left[ G_{h/2} \ast, \psi \right] V
\approx       
\left(\nabla \psi \cdot \nabla \right) G_{h/2} \ast V
\quad\text{or}\quad
\left[ G_{h/2} \ast, \psi \right] V
\approx       
h\nabla \psi \cdot \Brac{\nabla G_{h/2} \ast V}.
\Eeqn
\begin{proof} Expanding $\psi(x-z)-\psi(x) = -z\cdot\nabla\psi(x) 
+ O(|z|^2{\Norm{\nabla\psi}_\infty})$, we obtain
\begin{align*}
&  \frac1h \left[ G_{h/2} \ast, \psi \right] V\\
& =  \frac1h\int 
G_{h/2}(z)\Brac{\psi(x-z)-\psi(x)}V(x-z)
\,dz\\
& =  \frac1h\int 
G_{h/2}(z)\Brac{-z\cdot\nabla\psi(x) 
+ O(|z|^2{\Norm{\nabla\psi}_\infty})}V(x-z)
\,dz\\
& = 
\nabla\psi(x)\int\frac{-z}{h}G_{h/2}(z)V(x-z)\,dz + 
O\Brac{\int G_{h/2}(z)\frac{|z|^2}{h}
\Norm{\nabla^2\psi}_\infty \Abs{V(x-z)}\,dz}
\\
& =  
\nabla\psi(x)\cdot(\nabla G_{h/2}*V)(x)
+ O\Brac{\Norm{\nabla^2\psi}_\infty\Brac{\frac{|z|^2}{h}G_{h/2}(z)}*\Abs{V}},
\end{align*}
which is precisely the statement of the lemma.
\end{proof}



\subsection{Proof of the lemmas}
We first give the proof of Lemma \ref{lemma MM} illustrating the 
minimizing movements interpretation of the thresholding scheme.
\begin{proof}[Proof of Lemma \ref{lemma MM}]
      We will only prove the localized version \eqref{MM local}
as the global minimization property \eqref{MM} follows by choosing
$\psi\equiv 1$. The energy-dissipation estimate \eqref{energy estimate} 
follows by successively comparing the functional values for 
\eqref{MM} evaluated at 
$u^n$ and $u^{n-1}$.

      We first note that the combination of convolution and thresholding $u^{n} := \frac{G_h\ast u^{n-1}}{|G_h\ast u^{n-1}|}$ is equivalent to maximizing
      $(u\cdot G_h\ast u^{n-1})(x)$ \Em{pointwise at each $x$} among all
      $u\colon \domain \to \R^2$ with $|u|\leq 1$.
(This simply follows from the fact that for each
$a\neq 0\in\mathbb{R}^2$, $\hat{a}=\frac{a}{\abs{a}}$ is the unique
maximizer of
$a\cdot b$ among $\abs{b}\leq 1$.)
Therefore, for any non-negative function $\psi\geq0$, $u^n$ minimizes the linear functional
      \[
       -\frac2h \int \psi \, u \cdot G_h\ast u^{n-1}\,dx.
      \]
      Using
      \begin{align*}
       -2u\cdot G_h\ast u^{n-1} = &- u \cdot G_h\ast u + ( u-u^{n-1}) \cdot  G_h\ast( u-u^{n-1})\\ & - u^{n-1}\cdot G_h \ast u^{n-1}
       +  u^{n-1}\cdot G_h \ast u -  u\cdot G_h \ast u^{n-1},
      \end{align*}
we see that $u^n$ effectively minimizes the following functional
\begin{multline*}
\frac1h\int \psi(1- u \cdot G_h\ast u)
+
\frac1h\int\psi( u-u^{n-1}) \cdot  G_h\ast( u-u^{n-1})\\
+  \frac1h\int \psi(u^{n-1}\cdot G_h \ast u -  u\cdot G_h \ast u^{n-1}).
\end{multline*}
Note that by \eqref{selfadj}, we can write the last term of the above
as
\[
\frac1h\int u\cdot G_h*(\psi u^{n-1}) -  u\cdot \psi G_h \ast u^{n-1})
= \frac1h\int u\cdot [G_h*, \psi] u^{n-1}.
\]
Subtracting the irrelevant term
\MathSty{\frac1h\int u^{n-1}[G_h*, \psi]u^{n-1}} to the minimization
gives exactly expression \eqref{MM local}.

      Finally, note that with $\psi\equiv 1 > 0$ and tracing back our
      steps we see that thresholding is indeed \emph{equivalent} to solving the global minimization problem \eqref{MM}.
    \end{proof}

Now we continue to the proofs of the other technical lemmas.

      \begin{proof}[Proof of Lemma \ref{lem energy 1}]
	Statement (i)\eqref{E finite diff} 
	follows from a direct computation. Indeed, due to the normalization $\int G(z)\,dz =1$, we have
	\[
	E_h(u,\phi) = \frac1{2h} \int \psi(x)\int G(z) \, 2\,\big(1-u(x)\cdot u(x-\h z)\big) \,dx\,dz.
	\]
	Then \eqref{E finite diff} follows from the identity
	$2(1-u\cdot v) = |u-v|^2 $ valid for any pair of unit vectors 
	$u$ and $v$.

For (i)\eqref{approx Dir}, note that 
	for $u\in W^{1,2}$, the finite differences 
	$\frac{u(x)-u(x-\h z)}{\h}$ in the representation 
	\eqref{E finite diff} of $E_h$ converge to the
	directional derivative $\left(z\cdot \nabla\right) u(x)$ pointwise almost everywhere. Thus we obtain
	by Fatou's lemma
	\begin{equation}\label{liminf preliminary}
	\frac12 \int  \psi(x) \int G(z) \left|\left(z\cdot \nabla \right)u(x) \right|^2dz\,dx \leq \liminf_{h\to 0} E_h(u,\psi).
	\end{equation}
	(Alternatively, it is also straightforward to see that the finite differences converge weakly in $L^2(G(z)dz\,dx)$, which clearly implies \eqref{liminf preliminary}.)
	Next we compute the inner integral explicitly:
\begin{align*}
  \int  G(z)  \left| (z \cdot \nabla) u(x)\right|^2 dz
& =  \sum_{i=1}^d \Abs{\partial_{x_i}u(x)}^2\int G(z)z_i^2\,dz\\
&= \sum_{i=1}^d \Abs{\partial_{x_i}u(x)}^2\iint G^1(z_1)G^{d-1}(z')z_1^2\,dz_1\,dz'
= 2\Abs{\nabla u(x)}^2,
\end{align*}
where we have used the symmetry and factorization property of the kernel 
\eqref{semigroup-factor}.
Therefore, we obtain
\begin{equation}\label{liminf Eh}
\int \psi(x) \left|\nabla u(x)\right|^2 dx \leq  \liminf_{h\to 0} E_h(u,\psi)\quad \text{for all non-negative test functions }\psi.
\end{equation}
	
	Therefore, by linearity in $\psi$, it suffices to prove the statement with $\psi\equiv 1$.
	Note that this will imply the \emph{strong} convergence of the difference quotients in $L^2(G(z)dz\,dx)$.
	
	An application of the fundamental theorem of calculus and the translation invariance of $\int\,dx$ yield
	\begin{align*}
	E_h(u)
	&= \frac12 \iint G(z) \int_0^1\int_0^1 \left(z\cdot \nabla \right) u(x+s\h z) \cdot \left(z\cdot \nabla \right) u(x+t\h z) \,ds\,dt \,dz\,dx\\
	&=\frac12 \iint G(z)  \left(z\cdot \nabla \right) u(x) \otimes z \colon \int_0^1\int_0^1  \nabla u(x+(t-s)\h z) \,ds\,dt \,dz\,dx.
	\end{align*}
	Since $ \left(z\cdot \nabla \right) u(x) \otimes z\in L^2(G(z)dz\,dx)$ and furthermore 
	\[
	\int_0^1\int_0^1  \nabla u(x+(t-s)\h z) \,ds\,dt  \rightharpoonup 
	\nabla u(x) \quad \text{in }L^2(G(z)dz\,dx),
	\]
	we obtain 
	\[
	\lim_{h\to0} E_h(u) =\frac12 \iint G(z) \left|\left(z\cdot \nabla \right) u(x)\right|^2 dz\,dx.
	\]
Therefore we do have
\[
\lim_{h\to0} E_h(u) = \int \left|\nabla u \right|^2dx,
\]
and by linearity and the lower-semicontinuity we obtain \eqref{approx Dir}.
	
(ii) follows from (i) and Jensen's inequality. 
Indeed, we may rewrite $E_{N^2 h}$ as
\[
  E_{N^2 h}(u,\psi)  \overset{(i)}{=} \frac12\iint  G(z)\,\psi(x)\, \Big| \frac1N \sum_{n=1}^N \frac{u(x-(n-1)\h z)-u(x-n\h z)}{\h}\Big|^2 dz\,dx,
\]
and by Jensen's inequality the integrand can be estimated:
	\begin{align*}
	  &\Big| \frac1N \sum_{n=1}^N \frac{u(x-(n-1)\h z)-u(x-n\h z)}{\h}\Big|^2 \\
	  &\qquad \qquad \qquad \qquad\qquad\leq  \frac1N \sum_{n=1}^N \Big|\frac{u(x-(n-1)\h z)-u(x-n\h z)}{\h}\Big|^2.
	\end{align*}
	By the translation invariance of $\int\,dx$ in the case of $\psi\equiv 1$ we obtain the monotonicity 
	\begin{equation}\label{monotonicity global}
	 E_{N^2 h}(u) \leq E_{h}(u).
	\end{equation}
        For non-constant $\psi\geq0$, we obtain
	\begin{align*}
	 E_{N^2 h}(u,\psi) 
	 &\leq \frac12 \iint G(z) \left( \frac1N \sum_{n=1}^N \psi\big(x+(n-1)\h z\big) \right) \Big|\frac{u(x)-u(x-\h z)}{\h}\Big|^2 dx\,dz\\
	 &\leq \frac12 \iint G(z) \left( \psi(x) + \|\nabla \psi\|_\infty N \h |z|\right) \Big|\frac{u(x)-u(x-\h z)}{\h}\Big|^2 dx\,dz.
	\end{align*}
	Using $G(z)|z| \lesssim G_{4}(z)$ and the monotonicity \eqref{monotonicity global} for the error-term we obtain \eqref{monotonicity}.
	
	For (iii), we first observe that by $\int \nabla G_h(z)\,dz =0$ and Cauchy-Schwarz we have
	\begin{align}\label{smuggeling in the average}
	  \int \psi \left| \nabla G_{h/2} \ast u \right|^2 dx 
	&=  \int \psi \left| \int \nabla G_{h/2}(z) \left(u(x) - u(x- z) \right)dz\right|^2  dx\\
	&\leq  \int \psi  \left( \int \left|\nabla G_{h/2}(z) \right| dz\right)\notag
	\left( \int \left|\nabla G_{h/2}(z)\right| \left|u(x) - u(x- z) \right|^2 dz\right) dx.
	\end{align}
	Using the integral estimate $ \int \left|\nabla G_{h/2}(z) \right| dz \lesssim \frac1\h$ for the first inner integral and the pointwise estimate
	$\left|\nabla G_{h/2}(z)\right| \lesssim \frac1\h G_{h}(z)$ for the second one, we obtain
	\[
	    \int \psi \left| \nabla G_{h/2} \ast u \right|^2 dx 
	    \lesssim E_{h}(u,\psi).\qedhere
	\]
    \end{proof}

    \begin{proof}[Proof of Lemma \ref{lem energy 2}]
    In order to verify the monotonicity of the energy let us compute $E_h$ in Fourier space, i.e., in terms of the Fourier coefficients $\hat u(k)$ of $u$.
Indeed, since $|u|\equiv 1$ and by Plancharel we have
      \[
       E_h(u) = \frac1h \int u\cdot \left( u - G_h \ast u\right) dx = \frac1h \sum_{k\in \Z^3} \bar{\hat u}(k) \cdot \left(\hat u -  \widehat{G_h\ast u}\right)(k) 
       = \frac1h \sum_{k\in \Z^3} \big(1-\widehat G_h(k)\big) |\hat u(k)|^2,
      \]
      where the Fourier coefficients (note that $G_h$ is not periodic) of the kernel $G_h$ are given by 
      \[
      \widehat G_h(k) = \int_{\R^3} G_h(z) \,e^{-2\pi i\, k\cdot z} dz  
       = \exp\Big(-4\pi^2 h |k|^2\Big) .
      \] 
      Therefore, we may simply compute the derivative of $E_h$:
      \[
        \frac{d}{dh} E_h(u) = \sum_{k\in \Z^3} \frac{\partial}{\partial h} \left[\frac1h \left(1-\widehat G_h(k)\right)\right] |\hat u(k)|^2.
      \]
      Since
      \[
       \frac{\partial}{\partial h} \left[\frac1h \left(1-\widehat G_h\right)\right]
       =- \frac1{h^2} \left(1-\widehat G_h\right) - \frac1h \partial_h \widehat G_h
       = -\frac1{h^2} \left( 1+ h\partial_h \widehat G_h - \widehat G_h \right),
      \]
      it is enough to check whether
      \begin{equation}\label{fourier positive}
       \left( 1+ h\partial_h \widehat G_h - \widehat G_h \right) \geq 0.
      \end{equation}
      To do so, we write $s:= \Big(\frac{2\pi}{\Lambda}\Big)^2 h |k|^2$.
Then the above holds due to the fact that $e^s \geq 1 + s$ for 
all $s \geq 0$. This concludes the argument for \eqref{dhE}.

          Computing the derivative in \eqref{dhE} in physical space, 
i.e., in terms of $u$ instead of $\hat u$, we obtain
      \[
        \frac{d E_h}{dh} = \int \partial_h \left[ \frac1h \left( 1-u\cdot G_h\ast u \right) \right] dx.
      \]
      Since $G_h$ is the fundamental solution of the heat equation, cf.\ \eqref{G solves heat eqn}, we compute
      \[
       \partial_h \left[ \frac1h \left( 1-u\cdot G_h\ast u \right) \right] = -\frac1{h^2} \left( 1-u\cdot G_h\ast u \right) - \frac1h u \cdot \Delta G_h \ast u.
      \]
      Therefore, using the semi-group property of $G$ and the anti-symmetry of its gradient $\nabla G$  we obtain
      \[
         -\frac1h E_h(u) - \frac1h \int u \cdot \Delta G_h \ast u \,dx = -\frac1h E_h(u)  + \frac1h \int |\nabla G_{h/2}\ast u|^2 dx,
      \]
      which is precisely \eqref{D and E sharp}.
      
    \end{proof}
    
      \begin{rem}
      An alternative approach to the proof above is to use directly the energy dissipation relation of the energy $\frac12 \int |v|^2\,dx$ for the heat equation. Combined with $|u|=1$ we obtain
      \begin{equation}\label{E is average of dirichlet}
      E_h(u) = \frac2h \int_0^{h/2} \int \left|\nabla G_{t}\ast u\right|^2 dx\,dt.
      \end{equation}
      This means that the thresholding energy $E_h$ is nothing but an average of (twice the) Dirichlet energies along the heat flow. By the energy dissipation relation of the Dirichlet energy for the heat equation we obtain in particular
      \[
      \int \left|\nabla G_{h/2}\ast u\right|^2 dx \leq E_h(u)
\leq \int \left|\nabla u\right|^2 dx.
      \]
      Note that in our case, the second inequality is empty because our state variable $u$ is not in $W^{1,2}$.
      In view of \eqref{E is average of dirichlet}, the reverse inequality of \eqref{D and E easy}, namely an estimate of $E_h(u)$ in terms of the Dirichlet energy of $G_{h/2}\ast u$, seems not obvious.
       However, there are two simple cases. If $u\in W^{2,2}$, such an estimate is available. The second easy example is the vector field $u^0$ defined in \eqref{construct u0}. It is easy to check that such an estimate is available as well for $u^0$.
      \end{rem}
    
    \begin{proof}[Proof of Lemma \ref{lem estimate dtu}]
      Using the triangle and Young's inequalities, we have for any two vector fields $u, v$
      \[
        |u-v|^2 \lesssim \left|G_{h} \ast (u-v)\right|^2 + | G_{h}\ast u - u|^2 + |G_{h}\ast v - v|^2.
      \]
      If additionally $|u|\equiv 1$, we have
      \[
       | G_{h}\ast u - u|^2 = 2(1-u\cdot G_h\ast u) - (1-G_h\ast u \cdot G_h\ast u).
      \]
      Therefore when applying this to $u=u^n$ and $v=u^{n-1}$, the symmetry of the kernel $G_h$ implies
      \[
       \int_0^T\int \big|\partial_t^h u^h\big|^2\,dx\,dt
       \lesssim \int_0^T\int \big| G_h\ast \partial_t^h u^h\big|^2 \,dx\,dt
       +2 \sum_{n=0}^N \left( E_h(u^n)- E_{2h}(u^n) \right),
      \]
      which by the energy-dissipation estimate \eqref{energy estimate} yields the claim.
    \end{proof}

   Next we turn to the derivation of the two Euler-Lagrange equations for $u$.
    \begin{proof}[Proof of Lemma \ref{lem EL eq}(i)]
      We first prove \eqref{EL inner}, the Euler-Lagrange equation coming from \emph{inner} variations $u_s$ defined in \eqref{inner variations}.
      From the minimality \eqref{MM}, we obtain
      \begin{equation}\label{EL before computation}
       \frac{d}{ds}\Big|_{s=0} \left( E_h(u^n_s)
	  + \frac1h \int \left( u^n_s-u^{n-1}\right) \cdot  G_h\ast\left( u^n_s-u^{n-1}\right) dx \right) =0.
      \end{equation}

      We begin by computing the first variation $\frac{d}{ds} E_h(u_s^n)$ of the energy $E_h$, which will give us the right-hand side of \eqref{EL inner}.
      We work on a fixed time slice and drop the superscript $n$ for a cleaner notation. Note that
      \[
        \frac{d}{ds}\Big|_{s=0} E_h(u_s)
        = \frac1h \int u\cdot G_h \ast \left( \xi \cdot \nabla u\right) + \xi \cdot \nabla u \cdot G_h \ast u \,dx.
      \]
Since 
$\xi \cdot \nabla u = \nabla \cdot \left( \xi \, u \right) - \left(\nabla \cdot \xi \right) u$ 
and by the symmetry property \eqref{selfadj} of $G_h\ast$, we obtain
      \begin{align*}
       \frac{d}{ds}\Big|_{s=0}E_h(u_s)
       = & \frac1h \int u \left[ \nabla G_h \ast, \xi\cdot \right]  u \, dx
         - \frac2h \int \left(\nabla \cdot \xi \right) u \cdot G_h\ast u\,dx.
      \end{align*}
Now we claim that
      \begin{multline}\label{claim dE}
	\frac1h \int u \left[ \nabla G_h \ast, \xi\cdot \right]  u \, dx\\
= 2 \int \nabla \xi \colon u \nabla^2G_h \ast u \,dx +
	\frac1h \int \left(\nabla \cdot \xi \right) u \cdot G_h\ast u\,dx 
+ O\left(\Norm{\nabla^2\xi}_\infty\h E_h(u)\right).
      \end{multline}
This is intuitively correct, since we expect
      \[
       \left[ \nabla G_h \ast, \xi\cdot \right]  u  \approx \nabla \xi \colon (-z\otimes \nabla G_h(z)) \ast u
      \]
      so that we formally have
      \[
       \frac1h \int u \left[ \nabla G_h \ast, \xi\cdot \right]  u \, dx 
       \approx  \int \nabla \xi \colon u \,(-\frac z h \otimes \nabla G_h(z)) \ast u\, dx.
      \]
      Note that the kernel on the right may be rewritten as $-\frac z h \otimes \nabla G_h = \frac1h G_h \,Id + 2 \nabla^2 G_h$, which concludes the
      formal reasoning for \eqref{claim dE}.

Granted \eqref{claim dE}, we have
\Beqn\label{claim dE 2}
\frac{d}{ds}\Big|_{s=0}E_h(u_s)
= 2 \int \nabla \xi \colon u \nabla^2G_h \ast u \,dx 
-
	\frac1h \int \left(\nabla \cdot \xi \right) u \cdot G_h\ast u\,dx 
+ o(1)
\Eeqn
which is essentially the right hand side of \eqref{EL inner}, modulo
an integration by parts.
      
      In order to make \eqref{claim dE} rigorous, 
we rewrite the integral on its left-hand side as
      \begin{align*}
        \frac1h \int u \left[ \nabla G_h \ast, \xi\cdot \right]  u \, dx
	 =  \frac1{2h}  \iint \nabla G_h (z) \cdot \left(\xi(x) - \xi(x-z)\right) \left(  -2 u(x) \cdot u(x-z)\right) dz \, dx
      \end{align*}
Since $ |u(x)|^2,\,\,|u(x-z)|^2 \equiv 1$ so that 
$2-2u(x)u(x-z) = \abs{u(x)-u(x-z)}^2$ and 
\[
\iint \nabla G_h (z) \cdot \left(\xi(x) - \xi(x-z)\right)\,dz\,dx = 0,
\]
we have
      \begin{eqnarray*}
& & \frac1h \int u \left[ \nabla G_h \ast, \xi\cdot \right]  u \, dx\\
&=& \frac1{2h}
\iint \nabla G_h (z) \cdot \left(\xi(x) - \xi(x-z)\right) \left|u(x) - u(x-z)  \right|^2 dz \, dx\\
&=& \frac1{2h}  \iint \nabla G_h (z) \cdot 
\left(z\cdot \nabla \xi(x) +O(\Norm{\nabla^2\xi}_\infty|z|^2)\right)
\left|u(x) - u(x-z)  \right|^2 dz \, dx.
      \end{eqnarray*}
The above integral splits into two contributions: 
\begin{enumerate}
\item
The one coming from the first-order term $z\cdot \nabla \xi(x)$ simplifies to
      \[
	\frac1{2h}  \iint \nabla \xi \colon z\otimes \nabla G_h (z) (-2 u(x)\cdot u(x-z))\, dz \, dx,
      \]
      because here again the terms including $|u(x)|^2\equiv 1$ or 
$|u(x-z)|^2\equiv 1$ vanish identically. 
As in the conclusion of the above formal argument, we obtain
      the leading-order term in the expansion \eqref{claim dE}.

\item      
The second-order term $O(|z|^2)$ in the expansion of the test vector field 
$\xi$ is negligible as $h\to0$. 
Indeed, since 
\[
\Abs{\frac{|z|^2}{h}\nabla G }
\lesssim \Abs{\frac{|z|^3}{h^2}G_h}
\lesssim \frac1{\sqrt{h}}G_{2h}, 
\]
we have 
      \begin{align*}
       &\frac1{h}  \iint |z|^2|\nabla G_h (z)| \left|u(x) - u(x-z)  \right|^2 dz \, dx\\
       \overset{\phantom{\eqref{E finite diff}}}{\lesssim} 
&\frac1\h \iint G_{2h} (z) \left|u(x) - u(x-z)  \right|^2 dz\,dx
      \overset{\eqref{E finite diff}}{ \sim} \h E_{2h}(u) \overset{\eqref{dhE}}{\leq} \h E_h(u).
      \end{align*}
\end{enumerate}
      This concludes the rigorous justification of \eqref{claim dE}.

Going from \eqref{claim dE 2} to  the symmetrized form on the right-hand side of \eqref{EL inner} requires another manipulation which we provide now. 
To this end, we rewrite
\begin{equation}\label{symmetrize}
2\int\nabla\xi:u\nabla^2 G_h*u\,dx 
=2\int \nabla \xi \colon  G_{h/2} \ast u\,\nabla^2 G_{h/2} \ast u\,dx
+ 2\int \left[ G_{h/2}\ast,\nabla \xi \right] u \, \nabla^2 G_{h/2}\ast u\,dx,
\end{equation}
where we have
used the symmetry and semi-group properties \eqref{semigroup-factor} of the kernel.
The leading-order term is the first right-hand side term, which after integration by parts equals
\[
-2\sum_{i,j} \int \partial_i \xi_j  \partial_i G_{h/2}\ast u\,\partial_j G_{h/2}\ast u\,dx
- 2\int \Delta \xi \cdot  G_{h/2} \ast u\,\nabla G_{h/2} \ast u\,dx.
\]
Here, the first term is the desired term in the Euler-Lagrange equation \eqref{EL inner}; the second one is of lower order: since by $G_{h/2} \ast u \,\nabla G_{h/2} \ast u = \nabla (\frac12 |G_{h/2}\ast u|^2)$, we have
\[
- 2\int \Delta \xi \cdot  G_{h/2} \ast u \,\nabla G_{h/2} \ast u\,dx 
= 2\int \Delta (\nabla \cdot \xi ) \frac12 \left|G_{h/2}\ast u\right|^2 dx,
\]
which by $\int \Delta (\nabla \cdot \xi ) \,dx =0$ is equal to
\[
-\int \Delta (\nabla \cdot \xi ) \left(1- \left|G_{h/2}\ast u\right|^2\right) dx.
\]
Therefore, the first right-hand side term of \eqref{symmetrize} and the right-hand side of the Euler-Lagrange equation \eqref{EL inner} indeed agree to leading order:
\[
\left| - 2\int \Delta \xi \cdot  G_{h/2} \ast u \,\nabla G_{h/2} \ast u\,dx \right|
\lesssim \left\|\nabla^3 \xi \right\|_\infty  \int \left(1- \left|G_{h/2}\ast u\right|^2\right) dx 
=  \left\|\nabla^3 \xi \right\|_\infty h E_h(u) =o(1),
\]
where we have
used the symmetry and semi-group properties \eqref{semigroup-factor} of the kernel once more.

Now we turn to the second right-hand side term of \eqref{symmetrize}. 
By the commutator estimate \eqref{comm V}, we obtain
	\begin{align*}
	\int \left[ G_{h/2}\ast,\nabla \xi \right] u \, \nabla^2 G_{h/2}\ast u\,dx 
	=&h\sum_{i,j,k} \int \partial_i\partial_j \xi_k \colon  \nabla_i G_{h/2}\ast u \, \partial_j\partial_k G_{h/2}\ast u \,dx \\
	&+ O\Big(\|\nabla^3\xi\|_\infty h \int \left|\nabla^2 G_{h/2}\ast u \right|dx\Big).
	\end{align*}
	Note that the second right-hand side term vanishes as $h\to0$. Indeed, using Jensen's inequality, we have
	\[
	 \|\nabla^3\xi\|_\infty h\int \left|\nabla^2 G_{h/2}\ast u \right| dx
	  \lesssim \|\nabla^3\xi\|_\infty \h\left(  h\int \left|\nabla^2 G_{h/2}\ast u \right|^2 dx  \right)^\frac12.
	\]
	Exploiting $\int \nabla ^2 G_{h/2}\,dz =0$ and repeating the argument 
	 \eqref{smuggeling in the average} with $\h \nabla^2 G_{h/2}$ instead of $\nabla G_{h/2}$, we obtain
	 \[
	 \|\nabla^3\xi\|_\infty  \h \sqrt{E_h(u)} =o(1).
	 \]
	By symmetry of the second derivatives of $\xi$, and integrating by parts, we obtain
	\[
	h\sum_{i,j,k} \int \partial_i\partial_j \xi_k   \partial G_{h/2}\ast u \, \partial_j\partial_k G_{h/2}\ast u \,dx
	= -\frac h2 \sum_{i,j,k} \int \partial_i\partial_j (\nabla \cdot \xi )  \partial_i G_{h/2}\ast u \, \partial_j G_{h/2}\ast u \,dx,
	\]
	which is controlled by
	\[
	\| \nabla^3 \xi \|_\infty  h \int  \left|\nabla G_{h/2}\ast u\right|^2 dx \stackrel{\eqref{D and E easy}}{\lesssim } \| \nabla^3 \xi \|_\infty  h E_h(u) =o(1).
	\]

      Next we turn to the first variation of the metric term:
      \[
       \frac{d}{ds}\Big|_{s=0} \frac1h \int \left( u^n_s-u^{n-1}\right) \cdot  G_h\ast\left( u^n_s-u^{n-1}\right) dx 
       = -2\int \frac{u^n-u^{n-1}}{h} \cdot G_h \ast \left( \xi \cdot \nabla u^n\right)dx.
      \]
      Using the semigroup property \eqref{semigroup-factor} of $G_t$,
the self-adjointness \eqref{selfadj} of $G*$, and the relation 
$\xi \cdot \nabla u = \nabla \cdot \left( \xi \, u \right) - \left(\nabla \cdot \xi \right) u$,
      we obtain 
      \begin{multline}
	-2\int \frac{u^n-u^{n-1}}{h} \cdot G_h \ast \left( \xi \cdot \nabla u^n\right)dx\\
	= -2\int G_{h/2}\ast\left(\frac{u^n-u^{n-1}}{h}\right) \cdot \left(\xi\cdot \nabla\right) G_{h/2} \ast u^n \,dx + \text{Err}
\label{EL var of dist}
\end{multline}
where
\Beqn
\text{Err}
= -  2\int G_{h/2}\ast\left(\frac{u^n-u^{n-1}}{h}\right)\cdot \left( \left[\nabla G_{h/2}\ast,\xi\cdot\right]
	u^n - G_{h/2}\ast \left(\left(\nabla \cdot \xi\right) u^n\right)\right) dx.
\Eeqn
Note that, after integration in time, \eqref{EL var of dist} is precisely the left hand side of
\eqref{EL inner} if we can indeed show that, the time integral of 
$\text{Err}$ converges to zero as $h\to 0$.

      We now show that this is indeed the case. Omitting the superscript $n$ for a moment and setting again $u^{z} := u(\, \cdot \,+z)$ we may rewrite the commutator
      \[
	\left[\nabla G_{h/2}\ast,\xi\cdot\right] u = \int \nabla G_{h/2}(z) \cdot \left( \xi^{-z} -\xi \right) \left(u^{-z}-u\right) dz + u \,G_{h/2}\ast (\nabla \cdot \xi).
      \]
      Estimating $|\xi^{-z}-\xi|\leq \|\nabla \xi\|_\infty |z|$ and collecting the two terms involving the divergence of the test vector field $\xi$ we obtain
      \[
       \left[\nabla G_{h/2}\ast,\xi\cdot\right] u - G_{h/2}\ast \left(\left(\nabla \cdot \xi\right) u\right)
       = - \left[ G_{h/2}\ast, (\nabla \cdot \xi ) \right] u + O\Big(\|\nabla \xi\|\int \frac{|z|^2}{h} G_{h/2}(z) |u^{-z}-u| dz\Big).
      \]
      Note that $\frac{|z|^2}{h} G_{h/2}(z) \lesssim G_h(z)$.
      Now we estimate the commutator on the right-hand side:
      \[
       \left| \left[ G_{h/2}\ast, (\nabla \cdot \xi ) \right] u \right| \leq \|\nabla^2 \xi\| \left(|z|\, G_{h/2}(z)\right)\ast |u| \lesssim \|\nabla^2 \xi\| \h.
      \]
      Therefore by Cauchy-Schwarz, we have the following estimate for
the time integral of $\text{Err}$,
\[
      \begin{split}
\int \text{Err}(t)\,dt
\lesssim & \left( \|\nabla \xi\|_\infty + \|\nabla^2 \xi\|_\infty\right)\times\\
       &\left( \iint \big|G_{h/2} \ast \partial_t^h u^h\big|^2 dx \,dt\right)^{\frac12}
       \left( \iint  \Big(h +  \int G_h(z) \big|u^{h}(x-z)-u^h(x)\big|^2 dz \Big)dx\,dt \right)^{\frac12}.
\end{split}
      \]
      Using the representation \eqref{E finite diff} of the energy and the energy-dissipation estimate \eqref{energy estimate}, we see that the error term is bounded by
      \[
       C \left( \|\nabla \xi\|_\infty + \|\nabla^2 \xi\|_\infty\right)\left(E_h(u^0)\right)^{\frac12} \left( hT (1 + E_h(u^0)\right)^{\frac12} \overset{\eqref{Eu0}}{=}O(\h \log h).
      \]
      Together with \eqref{claim dE},
      this concludes the proof of \eqref{EL inner}.
      \end{proof}

    \begin{proof}[Proof of Lemma \ref{lem EL eq}(ii)]
      First we prove the slightly different version
      \begin{equation}\label{EL outer strong}
	\left( Id - u^h\otimes u^h\right) \left( G_h\ast \partial_t^{-h}u^h -\Delta_h u^h\right) =0,
      \end{equation}
      where $\Delta_h$ is an approximation of the Laplacian, given by the following average of second differences:
      \begin{equation*}
	\Delta_h u(x) := \int G(z)  \frac{u(x+\h z) -2 u(x) + u(x-\h z)}{2h}\, dz.
      \end{equation*}
      Note that indeed $\lim_{h\to0} \Delta_h u = \Delta u $ for $u\in W^{2,2}$ and thus \eqref{EL outer strong} is the analogue of the classical equation 
      \begin{equation}\label{HMHF strong}
	(Id-u\otimes u) ( \partial_t u -\Delta u)=0.
      \end{equation}

      In order to derive \eqref{EL outer strong}, we start from the minimality \eqref{MM}, which yields
      \eqref{EL before computation} with $u_s^n$ replaced by $\tilde u_s^n$.
      We use the representation \eqref{E finite diff} of the energy to compute the first variation $ \frac{d}{ds} E_h(\tilde u_s^n)$, again drop the superscript $n$, and use the short-hand notation $u^z:= u(\,\cdot\,  +z)$ in the following computation:
      \[
        \frac{d}{ds}\Big|_{s=0} E_h(\tilde u_s) = \frac1h \int G_h(z) \int ( u-u^{-z})
        \cdot \left( (Id-u\otimes u) \varphi -(Id-u^{-z}\otimes u^{-z}) \varphi^{-z}
        \right) dx\,dz,
        \]
      which because of the translation invariance of $\int \, dx$ is equal to
      \[
       \frac1h \int G_h(z) \int \left[( u-u^{-z}) -(u^z-u)\right]
        \cdot (Id-u\otimes u) \varphi  dx\,dz = -2\int \Delta_h u \cdot (Id-u\otimes u) \varphi\,dx. 
      \]
      The first variation of the metric term is
      \[
        \frac{d}{ds}\Big|_{s=0} 
	   \frac1h \int \left( \tilde u^n_s-u^{n-1}\right) \cdot  G_h\ast\left( \tilde u^n_s-u^{n-1}\right) dx
	  = 2 \int  G_h\ast\left( \frac{u^n-u^{n-1}}h\right) \cdot  (Id-u^n\otimes u^n) \varphi\, dx.
      \]
      This yields the ``classical'' version \eqref{EL outer strong} of the second Euler-Lagrange equation
      \eqref{EL outer weak}.
      
      To obtain the ``weak'' form \eqref{EL outer weak} of the Euler-Lagrange equation \eqref{EL outer strong},
      we proceed as in the well-known case of harmonic map heat flow \cite{chen1989weak}.
      We briefly recall this idea described in Evans' book \cite[\S 5.1.1]{evans1990weak}.
      We follow this more general approach because it is a natural way to derive the equations for $u$ (or the phase) away from the filament, and furthermore because it directly generalizes to higher dimensions, which we will exploit in \S \ref{sec:hmhf}.
      
      The idea is to take the wedge-product of \eqref{EL outer strong} and $u$ which leads to cancellations of nonlinear terms involving derivatives.
      More precisely, testing the $i^{th}$ component of \eqref{EL outer strong} with the $j^{th}$ component $u_j$ times a test function $\zeta$ and subtracting
      the the same quantity with exchanged roles for $i$ and $j$ one obtains that a solution to
      \eqref{HMHF strong} solves
      \[
	\iint \left(u_j\partial_t u_i  - u_i \partial_t u_j\right) \zeta - \left(  u_j \Delta u_i - u_i \Delta u_j\right) \zeta \,dx\,dt=0
      \]
      and when integrating by parts in the last term, the terms $\nabla u_i \cdot \nabla u_j \,\zeta$ cancel and we obtain
      \begin{equation}\label{HMHF compensated compactness}
	\iint \left(u_j\partial_t u_i  - u_i \partial_t u_j\right) \zeta + \left(  u_j \nabla u_i - u_i \nabla u_j\right) \cdot \nabla \zeta \,dx\,dt=0
      \end{equation}
      for all test functions $\zeta$. Note that this formulation has the advantage of being compact in $W^{2,2}$.
      
      In our case of \eqref{EL outer strong}, we follow the same idea. Let us omit the superscript $h$ again for this computation. We obtain in the first step that our solution $u=u^h$ of \eqref{EL outer strong} satisfies
      \[
	\iint \left(u_j\partial_t^{h} G_h \ast u_i  - u_i \partial_t^{h} G_h \ast u_j\right) \zeta - \left(  u_j \Delta_h u_i - u_i \Delta_h u_j\right) \zeta \,dx\,dt=0,
      \]
      where $\Delta_h$ denotes the above mentioned approximation of the Laplacian and $\partial_t^{-h}$ the difference quotient backwards in time.
      The integration by parts gets replaced by the following discrete version.
      Writing again $u^z:= u(\,\cdot\, +z)$ we have
      \begin{align*}
       &\int \left(  u_j \Delta_h u_i - u_i\Delta_h u_j\right) \zeta \,dx\\
       &\qquad  =\frac1{2h} \int G_h(z) \int  u_j [(u_i^z-u_i)-(u_i-u_i^{-z})] \zeta -u_i [(u_j^z-u_j)-(u_j-u_j^{-z} )] \zeta
       \,dx\,dz,
      \end{align*}
      which by the translation invariance of $\int \, dx$ is equal to
      \begin{align*}
     &\frac1{2h} \int G_h(z) \int u_j (u_i^z-u_i) \zeta - u_j^z(u_i^z-u_i ) \zeta^z
      - u_i(u_j^z-u_j)\zeta + u_i^z (u_j^z-u_j) \zeta^z
       \,dx\,dz\\
      &\quad =
       \frac1{2h} \int G_h(z) \int (u_i^z-u_i) (u_j\zeta - u_j^z \zeta^z) -(u_j^z-u_j) (u_i\zeta - u_i^z \zeta^z) 
       \,dx\,dz.
     \end{align*}
     Writing $u_i\zeta - u_i^z \zeta^z = - u_i (\zeta^z-\zeta) - (u_i^z-u_i) \zeta^z$ (and the same for $u_j$ instead of $u_i$), the terms involving the correction $(u_i^z-u_i) \zeta^z$ cancel and we obtain
     \[
      \int \left(  u_j \Delta_h u_i - u_i \Delta_h u_j\right) \zeta \,dx
      =  -\frac1{2h} \int G_h(z) \int
      (u_i^z-u_i) u_j(\zeta^z - \zeta) - (u_j^z-u_j) u_i (\zeta^z - \zeta) 
       \,dx\,dz.
     \]
     Now we may replace the finite difference $\zeta^z-\zeta$ by the gradient, i.e., we expand 
     $\zeta(x+z)-\zeta(x) = z\cdot \nabla \zeta(x) +O(|z|^2)$.
     Since $\nabla G(z) = -\frac z2 G(z)$ is antisymmetric, this yields the first-order term
     \[
        -\int ( u_i\nabla G_h\ast  u_j - u_j\nabla G_h \ast u_i)\cdot \nabla \zeta
       \,dx.
     \]
     The second-order term is controlled by
     \[
      \int \left( \frac{|z|}{\h}\right)^2 G_h(z) \int |u^z-u|\,dx\,dz.
     \]
     Using Jensen's inequality and the relation $\left( \frac{|z|}{\h}\right)^2 G_h(z) \lesssim G_{2h}(z)$
     we obtain the bound
     \[
       \left(\int G_{2h}(z) \int |u^z-u|^2\,dx\,dz\right)^\frac12 \sim \left(h E_{2h}(u)\right)^\frac12
       \overset{\eqref{dhE}}{\leq} \left(h E_{h}(u)\right)^\frac12.
     \]
     By the energy-dissipation inequality \eqref{energy estimate}, this is of order 
     $(h |\log h| )^\frac12 \to0$, which
     concludes the proof of \eqref{EL outer weak}.
    \end{proof}
    
    \subsection{Proof of the localized energy dissipation
Proposition \ref{prop gronwall}}\label{sec:proof prop}
    
     
     We fix $\sigma>0$ sufficiently small such that \eqref{phi hessian} holds for all $t\in[0,T]$. This is possible since by assumption the flow $\Gamma_t$ is smooth. We set $\phi=\phi_\sigma$ and let w.l.o.g.\ $T=Nh$ for some $N\in \N$. 
     As in \cite{lin1998complex}, our aim is to derive a Gronwall-type inequality for the localized energies $E_h(u^h,\phi)$.
     For $n\in \N$, comparing $u^n$ to its predecessor $u^{n-1}$ in the localized minimizing movements interpretation \eqref{MM local} of Lemma \ref{lemma MM} with our specific test function $\phi^n:=\phi(nh)\geq0$  we obtain
     \begin{align*}
       \frac1h \left(E_h(u^n,\phi^n) - E_h(u^{n-1},\phi^{n-1}) \right)
       \overset{\eqref{MM local}}{\leq}& -\int \phi^n \frac{ u^n-u^{n-1}}{h} \cdot  G_h\ast\left( \frac{u^n-u^{n-1}}{h}\right) dx \\
     &- \int \frac{u^n-u^{n-1}}{h} \cdot \frac1h \left[ G_h \ast, \phi^n \right] u^{n-1} \,dx\\
     &+ E_h\big(u^{n-1},\frac{\phi^n-\phi^{n-1}}{h}\big).
     \end{align*}   
We now sum over $n$ from $1$ to $N$ and multiply by the time-step size $h$, to obtain
     \begin{align}
       E_h(u^h(T),\phi(T)) -E_h(u^h(0),\phi(0)) \leq I_1 + I_2 + I_3
\label{int dtEh}
\end{align}
where
\begin{align}
I_1 &=
 -h \sum_{n=1}^N \int \phi^n \frac{ u^n-u^{n-1}}{h} \cdot  G_h\ast\left( \frac{u^n-u^{n-1}}{h}\right) dx,\label{I_1}\\
     I_2 & = - h\sum_{n=1}^N \int \frac{u^n-u^{n-1}}{h} \cdot \frac1h \left[ G_h \ast, \phi^n \right] u^{n-1} \,dx,\label{I_2}\\
     I_3 & = h\sum_{n=1}^N E_h\big(u^{n-1},\frac{\phi^n-\phi^{n-1}}{h}\big).
\label{I_3}
     \end{align}  
     Next we manipulate these three integrals separately.

\Bf{Analysis of $I_3$.} 
This is the easiest term. Since 
$\phi^n-\phi^{n-1} = \int_{(n-1)h}^{nh} \partial_t \phi \,dt$,
the interpolation $u^h(t)$ is piecewise constant in time, and the energy 
$E_h$ is linear in the second argument,
     we may rewrite $I_3$ as
     \Beqn\label{I_3_sim}
      h\sum_{n=1}^N E_h\Big(u^{n-1},\frac{\phi^n-\phi^{n-1}}{h}\Big) = \int_0^T E_h(u^h(t),\partial_t \phi)\,dt.
     \Eeqn
The above form will be combined with other terms in the conclusion
section.
    
\Bf{Analysis of $I_1$.} 
We claim that this term is almost negative. 
Indeed, putting $V:= \frac{u^n-u^{n-1}}{h}$ and $\phi:= \phi^n$, 
we may rewrite each summand of $I_1$ as
     \begin{equation}\label{diss term comm}
     -\int \phi\, V \cdot  G_h\ast V dx
     =      -\int \phi \big|  G_{h/2}\ast V\big|^2  dx - \int G_{h/2}\ast V \cdot 
     \left[ G_{h/2} \ast, \phi\right] V\,dx.
     \end{equation}
     It is thus enough to estimate the time-integral of the second right-hand 
side integral of \eqref{diss term comm} which will be shown to be
an error term. Using \eqref{comm V}, we expand the commutator as
     \begin{align*}
       \left[ G_{h/2} \ast, \phi \right] V
       = &\Brac{\nabla \phi \cdot \nabla G_{h/2} \ast V + O\big(\|\nabla^2 \phi\|_\infty \frac{|z|^2}{h} G_{h/2} \ast |V|\big)}h.
     \end{align*}
     Hence the the contribution of the second right-hand side integral in \eqref{diss term comm}  to $I_1$
can be decomposed into two terms.
     The first-order term is
     \begin{eqnarray*}
    &&  h\sum_{n=1}^N \left|h \int \big(G_{h/2}\ast \partial_t^h u^h\big)
\nabla \phi^n 
\cdot \nabla \Big( G_{h/2}\ast \partial_t^h u^h\Big) dx\right|\\
& = &
    h\sum_{n=1}^N \left|\frac h2 \int \nabla \phi^n \cdot \nabla \Big( \big|G_{h/2}\ast \partial_t^h u^h\big|^2\Big) dx\right|
    = h\sum_{n=1}^N \left|\frac h2 \int \Lap \phi^n \Big( \big|G_{h/2}\ast \partial_t^h u^h\big|^2\Big) dx\right|\\
    & \lesssim &  
\Norm{\nabla^2\phi}_\infty h^2\sum_{n=1}^N \left|\int 
\big|G_{h/2}\ast \partial_t^h u^h\big|^2 dx\right|
\lesssim
\Norm{\nabla^2\phi}_\infty h\int _0^T\int
\Abs{G_{h/2}\ast \partial_t^h u^h}^2\,dx\, dt,
     \end{eqnarray*}
     while the second-order term to $I_1$ is estimated by $\|\nabla^2 \phi\|_\infty\leq C(\sigma)$ times the following 
     expression,
     \begin{align*}
 h^2 \sum_{n=1}^N&\int \Abs{G_{h/2}*\partial_t^h u^h}
\Abs{\Brac{\frac{\abs{z}^2}{h}G_{h/2}}*|\partial_t^h u^h|}\,dx\\
=& h \iint \Abs{G_{h/2}*\partial_t^h u^h}
\Abs{\Brac{\frac{\abs{z}^2}{h}G_{h/2}}*|\partial_t^h u^h|}\,dx\,dt\\
\leq &h
\left(\iint |G_{h/2}*\partial_t^h u^h|^2 dx\,dt\right)^\frac12
\left(\iint \left(\Brac{\frac{\abs{z}^2}{h}G_{h/2}}*|\partial_t^h u^h|\right)^2 dx \,dt\right)^\frac12.
     \end{align*}
Since  $\frac{|z|^2}{h}G_{h/2} \lesssim G_h$  and by the $L^2$-contraction property
of the heat kernel, the second right-hand side integral in $I_1$ 
can be estimated by
\[
\iint \left(\Brac{\frac{\abs{z}^2}{h}G_{h/2}}* |\partial_t^h u^h|\right)^2 dx \,dt
\lesssim \iint |\partial_t^h u^h|^2 dx \,dt.
\]
Combining the two estimates for the first- and second-order terms with the energy-dissipation estimate \eqref{energy estimate} and the estimate \eqref{dtu} on the time derivative of $u$, the contribution of the second right-hand side integral of \eqref{diss term comm} to \eqref{int dtEh}
     is controlled by
     \begin{align*}
	&\left|\int_0^T \int G_{h/2}\ast\partial_t^h u^h \cdot 
     \left[ G_{h/2} \ast, \phi \right] \partial_t^h u^h \,dx\,dt\right| \\
    & \leq  C(\sigma)\left(h E_h(u^0) + h \left( E_h(u^0)\right)^\frac12 \left(\left(1+\frac Th\right) E_h(u^0) \right)^\frac12 \right)
     \overset{\eqref{Eu0}}{\leq}  C(\sigma)(1+T)\h |\log h| \to 0
     \end{align*}
so that the second term of \eqref{diss term comm} is negligible as
$h\to 0$. To conclude, we have
\Beqn\label{I_1_sim}
I_1 = - h\sum_{n=1}^N 
\int\phi\Abs{G_{h/2}*\Brac{\frac{u^n-u^{n-1}}{h}}}^2\,dx
+ o(1).
\Eeqn

\Bf{Analysis of $I_2$.}
This is the leading-order and most difficult term 
in \eqref{int dtEh}. We first give a short formal argument as a motivation.
     Expanding the commutator as in \eqref{comm V}, or heuristically
as in \eqref{comm asym}, we obtain the leading-order term as
     \begin{align}
       -\iint \partial_t^h u^h \cdot \frac1h \left[ G_h \ast, \phi \right] u^h \,dx \,dt
      \approx 
        -2 \iint \partial_t^h u^h \cdot \left(\nabla\phi \cdot \nabla \right) G_{h}\ast u^{h}\,dx\,dt,\label{I2 heuristic}
     \end{align}
     which is roughly the left-hand side of the Euler-Lagrange equation \eqref{EL inner} with $\xi = \nabla \phi$.

     In order to make the above rigorous, we make use of some cancellations, in particular the fact that the vectors $\partial_t^h u^h$ and $u^h$ are almost orthogonal.
     
     First, we observe that because of the antisymmetry
$\int V [G_h\ast,\phi]V\,dx =0$ of the commutator,
we may write this integral as
     \[
     I_2 = -h\sum_{n=1}^N 
     \int\frac{u^n-u^{n-1}}{h} \cdot  \frac1h \left[ G_h \ast, \phi^n \right]\Big(\frac{u^n+u^{n-1}}{2}\Big)  \,dx.
     \]
     There are two cancellations in this integral we will take advantage of:
     \begin{itemize}
     \item To first order, the commutator behaves like $\nabla\phi \cdot\nabla G_h\ast \big(\frac{u^n+u^{n-1}}{2}\big)$, which improves the order of the kernel from $\frac1h$ with the commutator to $\frac1\h$ for the kernel of the leading-order term.
     \item The two vectors $\frac{u^n-u^{n-1}}{h}$, $\frac{u^n+u^{n-1}}{2}$ are orthogonal, and this is still approximately true after 
     the convolution with the heat kernel.
     \end{itemize}
Now we fix $n$ and write for simplicity $u:=u^n$, $v:=u^{n-1}$ and $\phi:= \phi^n$. Then we rewrite each summand in $I_2$ 
in the following more symmetric fashion:
     \begin{align}
     &\int \frac{u-v}{h} \cdot  \frac1h \left[ G_{h} \ast, \phi \right]\Big(\frac{u+v}{2}\Big) dx \label{I2.summand}\\
	     &=
	     \int G_{h/2}\ast\Big(\frac{u-v}{h}\Big) \cdot  \frac1h \left[ G_{h/2} \ast, \phi \right]\Big(\frac{u+v}{2}\Big) -
	     \frac1h \left[ G_{h/2} \ast, \phi \right]\Big(\frac{u-v}{h}\Big) \cdot G_{h/2}\ast \Big(\frac{u+v}{2}\Big)  \,dx.
	     \nonumber
	     \end{align}

Second, we will dissect carefully the commutators appearing above.
The computation is more elaborate than the simple first order asymptotics
\eqref{comm V}.
For any vector field $w$ and any smooth test function $\phi$,
we expand the commutator now as
	     \begin{align*}
	     \left[ G_{h/2} \ast, \phi \right] w 
	     =&  \int G_{h/2}(z) \left(\phi(\,\cdot\,-z)-\phi(\,\cdot\,)\right) w(\,\cdot\,-z) \,dz\\ 
	     =&   \Big(\Big(-  \partial_i\phi z_i  + \frac12 \partial_{ij}\phi 
				     z_i z_j  - 
				     \frac16 \partial_{ijk}\phi z_iz_jz_k\Big) G_{h/2}\Big) \ast  w + O\Big( |z|^4 G_{h/2}\ast |w|\Big),
	     \end{align*}
where we have summed over repeated indice.
Using the identities
\begin{align*}
- z_i G_{h/2} = h \partial_i G_{h/2},\quad  z_i z_j G_{h/2} = h^2 \partial_i\partial_j G_{h/2} - h \delta_{ij} G_{h/2},\\
 - z_iz_jz_k G_{h/2} = h^3 \partial_i\partial_j\partial_k G_{h/2} - h\big(\delta_{ij} z_k + \delta_{ik} z_j + \delta_{jk}z_i\big)G_{h/2},      
\end{align*}
and the equality of mixed partials, we obtain
\begin{align}
\frac1h& \big[ G_{h/2} \ast,\,\phi \big] w\nonumber\\
= \,\,&\left(\nabla \phi \cdot \nabla\right) G_{h/2}\ast w + \frac h2 \nabla^2 \phi \colon \nabla^2 G_{h/2}\ast w - \frac12 \Delta \phi\, G_{h/2}\ast w
\label{comm higher order}\\
       &+\frac {h^2}{6} \sum_{i,j,k} \partial_i\partial_j\partial_k\phi \, \partial_i\partial_j\partial_k G_{h/2}\ast w - \frac h 2 \left(\nabla \Delta \phi \cdot \nabla \right) G_{h/2}\ast w
      + O\Big(\frac{ |z|^4}{h} G_{h/2}\ast |w|\Big).\nonumber
\end{align}

In order to concentrate on the key issue, we will first write down the dominating terms and
show the negligibility of the error terms afterwards.
Using the above expansion of the commutator for 
the difference quotient $w=\frac{u-v}{h}$ or the average $w=\frac{u+v}{2}$, 
respectively, we have, up to leading order, 
\begin{align}
& \text{Term \eqref{I2.summand}}\nonumber\\
\approx &\int G_{h/2}\ast\Big(\frac{u-v}{h}\Big) \cdot  \left(\nabla \phi \cdot \nabla\right) G_{h/2}\ast \Big(\frac{u+v}{2}\Big)
-
\left(\nabla \phi \cdot \nabla\right) G_{h/2}\ast \Big(\frac{u-v}{h}\Big) \cdot G_{h/2}\ast \Big(\frac{u+v}{2}\Big)\,dx \label{toward EL err 1}\\
= & 2\int G_{h/2}\ast\Big(\frac{u-v}{h}\Big) \cdot  \left(\nabla \phi \cdot \nabla\right) G_{h/2}\ast \Big(\frac{u+v}{2}\Big) dx
+\int \Delta \phi \,G_{h/2}\ast \Big(\frac{u-v}{h}\Big) \cdot G_{h/2}\ast\Big(\frac{u+v}{2}\Big)\,dx
\label{toward EL err 2}\\
\approx& 2\int G_{h/2}\ast\Big(\frac{u-v}{h}\Big) \cdot  \left(\nabla \phi \cdot \nabla\right) G_{h/2}\ast \Big(\frac{u+v}{2}\Big)\,dx.
\label{toward EL err 3}
     \end{align}
Hence,
     \begin{align}
     I_2 \approx &-2h\sum_{n=1}^N \int G_{h/2}\ast\Big(\frac{u^{n}-u^{n-1}}{h}\Big) \cdot  \left(\nabla \phi^n \cdot \nabla\right) G_{h/2}\ast \Big(\frac{u^n+u^{n-1}}{2}\Big) dx\nonumber\\
      \approx& -2 \iint \!G_{h/2} \ast \partial_t^{-h}u^h \cdot \left(\nabla \phi \cdot \nabla \right) G_{h/2} \ast u^h \,dx\,dt 
\quad\quad\text{(since $u^n \approx u^{n-1}$).}
\label{toward EL err 4}
     \end{align}
     which is
     precisely the left-hand side of the Euler-Lagrange equation \eqref{EL inner} with $\xi=\nabla \phi$. 
     Hence we expect to have
     \Beqn\label{I_2_sim}
      I_2 =
      -\frac1h \iint \Delta \phi\, ( 1-u^h \cdot G_h\ast u^h) \,dx\,dt\\
        +2 \sum_{i,j}\iint \!\partial_i\partial_j \phi \,\partial_i G_{h/2}\ast u^h \cdot \partial_j G_{h/2} \ast u^h \,dx\,dt + o(1).
     \Eeqn
Note that writing $I_2$ as \eqref{toward EL err 4} is
essentially the same as what we first noted in \eqref{I2 heuristic},
modulo the splitting of $G_h*$ into two separate $G_{h/2}$, but now all the
intermediate approximations and errors arising from
\eqref{I2.summand} to \eqref{toward EL err 1}--\eqref{toward EL err 4}
are spelled out carefully.
     
     \Bf{Conclusion.} With the above analysis of $I_1, I_2$, and $I_3$, 
in particular, combining expressions \eqref{I_1_sim}, \eqref{I_2_sim},
and \eqref{I_3_sim}, we obtain, 
     \begin{align}
       E_h(u^h(T),\phi(T))
       \leq &E_h(u^0,\phi(0))  -\int_0^T \int \phi \big| G_{h/2}\ast\partial_t^h u^h\big|^2 dx\,dt \notag\\
     &+2 \int_0^T \sum_{i,j}\int  \partial_i \partial_j \phi\; \partial_i G_{h/2} \ast u^h \cdot \partial_j G_{h/2} \ast u^h\,dx\,dt\label{int dtEh2}\\
     &+ \int_0^T \frac1h\int \left(\partial_t \phi-\Delta \phi\right) \big(1- u^h\cdot G_h\ast u^h \big) \,dx\,dt +o(1). \notag 
     \end{align}
     Recall that the missing argument that the error is indeed $o(1)$ as $h\to0$ will be given shortly.
     
     Now we decompose the last two right-hand side integrals into their near- and far-field contributions corresponding to respectively the
regions
     \[
      A_\sigma := \{(x,t) \colon d(x,\Gamma_t) <\sigma\}
\quad\text{and}\quad
      A^c_\sigma := \{(x,t) \colon d(x,\Gamma_t) \geq\sigma\}.
     \]
     In the far-field region $A^c_\sigma$, we have $\phi\geq \frac12 \sigma^2$ and therefore $|\partial_t \phi -\Delta \phi|, |\partial_i\partial_j \phi| \leq C(\sigma) \,\phi$.
     This implies that we may estimate the far-field contribution to the right-hand side of \eqref{int dtEh2} by
     \begin{align*}
     & 2 \iint_{A^c_\sigma}  \partial_i \partial_j \phi\; \partial_i G_{h/2} \ast u^h \cdot \partial_j G_{h/2} \ast u^h\,dx\,dt
     +\frac1h \iint_{A^c_\sigma} \left(\partial_t \phi-\Delta \phi\right) \big(1- u^h\cdot G_h\ast u^h \big) \,dx\,dt\\ &\qquad \overset{\eqref{D and E easy}}{\leq} C(\sigma)  \frac1h \iint \phi \,\big( 1-u^h\cdot G_h\ast u^h \big)\,dx\,dt.
     \end{align*}

     Next we turn to the more interesting near-field contribution. In order to control the last two right-hand side integrals of \eqref{int dtEh2} over 
region $A_\sigma$, we use the expansion of the heat operator \eqref{phi expansion of heat operator} and the estimate 
     on the Hessian \eqref{phi hessian} to obtain
     \begin{align*}
       & 2 \iint_{A_\sigma}  \partial_i \partial_j \phi\; \partial_i G_{h/2} \ast u^h \cdot \partial_j G_{h/2} \ast u^h\,dx\,dt
     +\frac1h \iint_{A_\sigma} \left(\partial_t \phi-\Delta \phi\right) \big(1- u^h\cdot G_h\ast u^h \big) \,dx\,dt\\
       &\qquad\leq 2 \iint_{A_\sigma} \big|\nabla G_{h/2} \ast u^h\big|^2 dx\,dt+ \frac1h\iint_{A_\sigma} \left(-2 + C\phi\right) \big( 1-u^h\cdot G_h\ast u^h \big)\,dx\,dt.
     \end{align*}
     Note that $\phi$ ``cuts out'' the vorticity set $\Gamma_t$ and therefore we expect $E_h(u^h,\phi)$ to stay finite as $h\to0$. However, there are two competing diverging terms, namely the Dirichlet energy $2\int \big|\nabla G_{h/2} \ast u^h\big|^2 dx$ and 
     the thresholding energy $-2 E_h(u^h)$. Miraculously, the monotonicity \eqref{dhE} of the energy in the time step size, or more exactly 
its equivalent statement \eqref{D and E sharp} (in Lemma \ref{lem energy 2}) provides precisely the correct relationship between the two quantities. 

With the above observation, we now compute:
\begin{align}
& E_h(u^h(T),\phi(T))\nonumber\\
\leq &
E_h(u^0,\phi(0)) - \int_0^T \int \phi \big| G_{h/2}\ast\partial_t^h u^h\big|^2 dx\,dt
+\frac{C(\sigma)}{h}\int_0^T \int\phi\Brac{1- u^h \cdot G_h*u^h}\,dx\,dt
\nonumber\\
& + 
2 \iint_{A_\sigma} \big|\nabla G_{h/2} \ast u^h\big|^2 dx\,dt+ \frac1h\iint_{A_\sigma} \left(-2 + C\phi\right) \big( 1-u^h\cdot G_h\ast u^h \big)\,dx\,dt
\nonumber\\
\leq & 
E_h(u^0,\phi(0)) - \int_0^T \int \phi \big| G_{h/2}\ast\partial_t^h u^h\big|^2 dx\,dt
+\frac{C(\sigma)}{h}\int_0^T \int\phi\Brac{1- u^h \cdot G_h*u^h}\,dx\,dt\nonumber\\
&
+ 
2 \iint_{A_\sigma} \big|\nabla G_{h/2} \ast u^h\big|^2 dx\,dt 
-2 \iint_{A_\sigma} \frac1h\big( 1-u^h\cdot G_h\ast u^h \big)\,dx\,dt.
\label{last}
\end{align}
We can bound the line \eqref{last} from above by
\begin{align}
& 2 \iint \big|\nabla G_{h/2} \ast u^h\big|^2 dx\,dt 
-2 \iint_{A_\sigma} \frac1h\big( 1-u^h\cdot G_h\ast u^h \big)\,dx\,dt\nonumber\\
\leq & 
2\int_0^T E_h(u^h)\,dt
-2 \iint_{A_\sigma} \frac1h\big( 1-u^h\cdot G_h\ast u^h \big)\,dx\,dt
\quad\text{(by \eqref{D and E sharp})}\nonumber\\
\leq &
2 \iint_{A_\sigma^c} \frac1h\big( 1-u^h\cdot G_h\ast u^h \big)\,dx\,dt
\leq
\frac{C(\sigma)}{h} \iint\phi\big( 1-u^h\cdot G_h\ast u^h \big)\,dx\,dt.
\end{align}
The above finally gives
\[
      E_h(u^h(T),\phi(T))
       \leq E_h(u^0,\phi(0))  -\int_0^T \!\int \phi \big| G_{h/2}\ast\partial_t^h u^h\big|^2 dx\,dt 
     +C(\sigma) \int_0^T\! E_h(u^h,\phi)\,dt+o(1). 
\]
     A standard Gronwall-argument yields
     \eqref{gronwall}.\hfill \qed

\Bf{Analysis of the error terms.}
Now we estimate the errors coming from \eqref{toward EL err 1}--\eqref{toward EL err 4}.

\Em{Error in \eqref{toward EL err 4}.} 
This error is due to replacing $u^{n-1}$ by $u^n$ and is bounded by
\begin{align*}
& \left|h\iint \big(G_{h/2}*(\partial_t^{-h} u^h)\big)
\nabla \phi \cdot \nabla \big(G_{h/2} \ast \partial_t^{-h} u^h\big)dx
\,dt\right| \\
=&
\left|h\iint \nabla \phi \cdot \nabla \big(\big|G_{h/2} \ast \partial_t^{-h} u^h\big|^2\big)dx\,dt\right| \\
\lesssim 
&\|\nabla^2 \phi\|_\infty h\iint \big|G_{h/2} \ast \partial_t^{-h} u^h\big|^2 dx\,dt
\overset{\eqref{energy estimate}}{=} O(h |\log h|)
\to 0\quad\text{as}\quad h\to 0.
\end{align*}

\Em{Error in \eqref{toward EL err 3}.} 
This is caused by omiting the second term in \eqref{toward EL err 2}
which involves the Laplacian of $\phi$, $\Delta\phi$. We will show that
after integration in time, this term is negligible.
For this, we use the identity $(a-b)\cdot (a+b) = |a|^2 - |b|^2$ to see that
     \begin{align}
      &h \sum_{n=1}^N\int \Delta \phi^n \,G_{h/2}\ast \Big(\frac{u^n-u^{n-1}}{h}\Big) \cdot G_{h/2}\ast\Big(\frac{u^{n}+u^{n-1}}{2}\Big)  dx \label{finding EL step2}\\
      =\,\,&h \sum_{n=1}^N\int \Delta \phi^n \frac12 \left( \frac1h \left(1-\big| G_{h/2}\ast u^{n-1}\big|^2\right) - \frac1h \left(1-\big| G_{h/2}\ast u^n\big|^2\right) \right) dx\notag\\
      \leq\,\, &h \sum_{n=1}^N \|\partial_t \nabla^2 \phi\|_\infty\frac h2 \int  \frac1h \left(1-\big| G_{h/2}\ast u^{n}\big|^2 \right) dx \notag\\
      \,\,&+ \|\nabla^2\phi\|_\infty  \frac h2 \int \frac1h \left(1-\big| G_{h/2}\ast u^N\big|^2\right) + \frac1h \left(1-\big| G_{h/2}\ast u^0\big|^2\right) dx. \notag
     \end{align}
     Note that we used that the terms $1-|G_{h/2}\ast u|^2$ are non-negative. By the symmetry of $G$, we have
     \[
      \int \frac1h \big(1- |G_{h/2}\ast u|^2\big) dx = E_h(u)
     \]
     and hence, using the energy-dissipation estimate \eqref{energy estimate}, we obtain the bound
     \[
      T h \|\partial_t \nabla^2 \phi\|_\infty E_h(u^0) + \|\nabla^2\phi\|_\infty h E_h(u^0) \\
      \lesssim \left(\|\partial_t \nabla^2 \phi\|_\infty+\|\nabla^2\phi\|_\infty\right) (1+T) h |\log h|,
     \]
     which vanishes in the limit $h\to0$.

\Em{Error in \eqref{toward EL err 1}.} 
The error in this line comes
from omitting the higher-order terms in the expansion of the commutators
\eqref{comm higher order}.
     To this end, we integrate by parts all derivatives which are on the time-derivative $G_{h/2} \ast \big(\frac{u-v}{h}\big)$.
     The resulting terms are of the form
     \begin{align}\label{towards EL 2nd order1}
      \int  h\,P(\nabla)\nabla^2 \phi \; Q(\nabla)G_{h/2}\ast \Big(\frac{u+v}{2}\Big) \cdot G_{h/2}\ast \Big(\frac{u-v}{h}\Big) dx,
     \end{align}
     where the linear differential operators $P(\nabla)$, $Q(\nabla)$ are both of order $\leq 2$. For the terms for which there is no derivative falling 
onto $G_{h/2}\ast \Big(\frac{u+v}{2}\Big)$, we
     proceed as in the lines following \eqref{finding EL step2}. Therefore, we may assume that the polynomial $Q$ is either homogeneous of order $1$, or $2$.
     In these cases we estimate by Cauchy-Schwarz after integration in time:
     \begin{align*}
     &h \sum_{n=1}^N \int  h\,P(\nabla)\nabla^2 \phi^n \; Q(\nabla)G_{h/2}\ast \Big(\frac{u^n+u^{n-1}}{2}\Big) \cdot G_{h/2}\ast \Big(\frac{u^n-u^{n-1}}{h}\Big) dx\\
     & \lesssim h \| P(\nabla)\nabla^2 \phi\|_\infty 
     \left( h \sum_{n=0}^N \int \left| Q(\nabla)G_{h/2}\ast u^n\right|^2 dx \right)^{\frac12}
     \left( h \sum_{n=1}^N \int \left| G_{h/2}\ast \Big(\frac{u^n-u^{n-1}}{h}\Big)\right|^2 dx \right)^{\frac12}.
     \end{align*}
     By the energy-dissipation estimate, the last right-hand side term is estimated by $E_h(u^0)$ while the first right-hand side integral can be manipulated as follows. By our assumption on $Q$ we have $\int Q(\nabla) G_{h/2}(z) \,dz  =0$ and therefore by Jensen's inequality
     \begin{align*}
      \int \left| Q(\nabla)G_{h/2}\ast u\right|^2 dx = &\int \left| \int Q(\nabla)G_{h/2}(z) \left(u(x-z)-u(x)\right) dz\right|^2dx\\
       \lesssim & \int \left| Q(\nabla)G_{h/2}(z)\right| dz \int \left|Q(\nabla)G_{h/2}(z)\right| \left| u(x-z)-u(x)\right|^2\, dz\,dx.
     \end{align*}
     Since $Q$ is of degree $\leq 2$, we have the integral estimate $\int \left|Q(\nabla)G_{h/2}\right| dz \lesssim \frac1h$ and the pointwise estimate 
     $\left|Q(\nabla)G_{h/2}(z)\right| \lesssim \frac1h G_h(z)$.
Thus
      \begin{align}\label{towards EL 2nd order2}
      \int \left| Q(\nabla)G_{h/2}\ast u\right|^2 dx 
       \lesssim \frac1h E_h(u).
     \end{align}
     Plugging in this estimate and using the energy-dissipation estimate \eqref{energy estimate} once more, we obtain the bound
     \[
      h \| P(\nabla)\nabla^2 \phi\|_\infty \left( T \frac1h E_h(u^0)\right)^\frac12 \left(E_h(u^0)\right)^{\frac12} 
      \overset{\eqref{Eu0}}{\lesssim} \h \| P(\nabla)\nabla^2 \phi\|_\infty (1+T) \left|\log h\right|,
     \]
     which as before vanishes as $h\to0$.
     
The third-order term works in the same fashion---only that the differential operators $P$ and $Q$ add up to order $3$ instead of $2$. 
     This weakens the estimate \eqref{towards EL 2nd order2} by an order of $\frac1h$.
     But we have one more power of $h$ in the prefactor so that we obtain \eqref{towards EL 2nd order1} with the prefactor $h^2$ instead of $h$.
     
     Finally, we estimate the error coming from the fourth-order term in the expansion of the commutator by
     \begin{align*}
       &h \sum_{n=1}^N \int \left(\frac{|z|^4}{h} G_{h/2}\right) \ast \Big|\frac{u^n-u^{n-1}}{h}\Big| + \Big| G_{h/2}\ast \Big(\frac{u^n-u^{n-1}}{h}\Big)\Big| 
       \int\frac{|z|^4}{h} G_{h/2}(z)\,dz\, dx\\
       &\lesssim h \iint \big| \partial_t^h u^h \big| \,dx\,dt \lesssim h T^{\frac12} \left(\iint\big| \partial_t^h u^h \big|^2 \,dx\,dt \right)^\frac12 
       \overset{\eqref{dtu}}{\lesssim} \h (1+T) E_h(u^0) \to0.
     \end{align*}

     \medskip 
%
%
     
With the above, we have thus taken care of all the error terms, concluding
the proof of Proposition \eqref{prop gronwall}.

    \subsection{Proof of the main result Theorem \ref{thm filament}}\label{sec:proof thm}
    
    \noindent \emph{Step 1: Compactness.}
    
    From our localized energy-dissipation inequality of Proposition \ref{prop gronwall}, the relation between Dirichlet and thresholding energy of Lemma \ref{lem energy 1}, and the estimate on the initial data \eqref{Eu0} we obtain for any fixed $\sigma>0$
    \begin{equation}
      \sup_{t\in(0,T)} \int \phi_\sigma\big| \nabla G_{h/2}\ast u^h \big|^2dx + \int_0^T \int \phi_\sigma \, \big| G_{h/2}\ast \partial_t^h u^h \big|^2 dx\,dt \leq C(\sigma)+o(1).
    \end{equation}
    Therefore, we may extract a subsequence which converges weakly to a map $u \in H^1_{\text{loc}}(\domain \setminus \Gamma)$ in the sense that
    \begin{eqnarray}
     G_{h/2} \ast u^h \to u &&\text{in } L^2, \label{comp Gu}\\
     \nabla G_{h/2} \ast u^h \rightharpoonup \nabla u  &&\text{in }L^2_{\text{loc}}((\domain\times(0,T) ) \setminus \Gamma),\quad \text{and}\nonumber\\
     \partial_t^h G_{h/2} \ast u^h \rightharpoonup \partial_t u &&\text{in }L^2_{\text{loc}}((\domain\times(0,T) ) \setminus \Gamma).\nonumber
    \end{eqnarray}
    The convergence \eqref{comp Gu} upgrades to the $L^2$-convergence of $u^h$  since by Jensen's inequality and $G_{h/2}\lesssim G_h$ we have 
    \[
\begin{split}
     \int  |u-G_{h/2}\ast u|^2dx 
     & = \int \left| \int G_{h/2}(z) (u(x)-u(x-z))\,dz\right|^2 dx\\
     & \lesssim \int \int G_{h/2}(z) \big|u(x)-u(x-z)\big|^2dz\,dx
    \lesssim h E_{h}(u^h,\zeta) \lesssim h|\log h| \to 0.
    \end{split}
   \] 
    \noindent \emph{Step 2: Convergence of the vorticity set.}
    
    Using again Proposition \ref{prop gronwall} we obtain the convergence of the vorticity set.
    
    \noindent \emph{Step 3: Convergence of $u^h$ away from the vorticity set.}
    
    By \eqref{comp u}--\eqref{comp dtu} we may pass to the limit in the weak form of the Euler-Lagrange equation \eqref{EL outer weak} whenever the test function $\zeta$ localizes away from $\Gamma$.
    
%
%

    \section{Other variants of thresholding}\label{sec:pinning}
        
    In this section, we discuss variants of thresholding which are
motivated by the minimizing movements interpretation.
    
    \subsection{Extensions to Neumann and Dirichlet boundary conditions}
    For convenience and a cleaner presentation, up to now we have 
restricted ourselves to the simplest case of periodic boundary conditions.
    However, when working on a bounded domain $\Omega\subset\mathbb{R}^d$, 
it is more natural to consider Neumann or Dirichlet boundary conditions 
corresponding to
\MathSty{\frac{\partial u}{\partial n}\big|_{\partial\Omega} = 0}
or \MathSty{u\big|_{\partial\Omega} = \bar{u}}
for the state variable $u$.
Here $\bar{u}$ is some prescribed function on $\partial\Omega$. 
We will show that it is 
very easy to incorporate these boundary conditions while keeping the 
same heat kernel $G_h$ for the whole space. Hence numerical efficiency
and the appealing simplicity of the scheme can be maintained.
The main idea is to extend $u$ appropriately from $\Omega$ to $\mathbb{R}^d$.

    \medskip
    
    Let $\Omega\subset \R^3$ be a smooth bounded domain. In order to solve the equation with Neumann boundary conditions on $\partial \Omega$, we first rewrite 
the minimization problem \eqref{MM} in the following equivalent symmetrized form
   \begin{multline*} \min \bigg\{ \; \frac1h	 \int_{\R^3}\int_\domain G_h(x-y)  \left(1-u(x)\cdot u(y)\right)dx\,dy\\
   + \frac1h \int_{\R^3}\int_\domain G_h(x-y) \left(u-u^{n-1}\right)(x) \cdot  \left(u-u^{n-1}\right)(y)\,dx\,dy \bigg\},
   \end{multline*}
Note that the outer integral 
$\int_{\R^3} \,dy$ in \eqref{MM} can be replaced by $\int_\domain \,dy$ 
without changing the integral drastically as the kernel $G_h(x-y)$ decays 
exponentially in $\abs{x-y}$.
   Therefore a natural generalization of this minimization problem to 
the bounded domain $\Omega$ with Neumann boundary conditions is
   \begin{multline}
\label{MM Neumann}
 \min \Big\{ \frac1h	\int_\Omega \int_\Omega G_h(x-y)  \left(1-u(x)\cdot u(y)\right)\,dx\,dy\\
  + \frac1h \int_\Omega\int_\Omega G_h(x-y) \left(u-u^{n-1}\right)(x) \cdot  \left(u-u^{n-1}\right)(y)\,dx\,dy \Big\},
    \end{multline}
    where the minimum runs over all vector fields $u\colon \Omega \to \R^2$ with $|u|\leq1$ a.e.\ in $\Omega$.

This may be interpreted as extending $u$ by zero off $\Omega$. Note that $u=0$ has equal distance to all points on the sphere $\mathbb{S}^1$ and therefore does not prefer any of these values. Another way to interpret the minimization problem \eqref{MM Neumann} is that there is no interaction with points outside the domain $\Omega$ and since no boundary conditions are enforced in the minimization procedure, it is reasonable to expect the minimizer to attain natural boundary conditions, i.e., Neumann boundary conditions.

   Note that for any unit vector field $u\colon\Omega\to\mathbb{S}^1$, 
the first term in \eqref{MM Neumann} can also be written as a weighted average
of finite differences,
   \begin{equation}\label{E Omega}
    E_{\Omega,h}(u):= \frac1h\int_\Omega \int_\Omega G_h(x-y) \left(1-u(x)\cdot u(y)\right)dx\,dy
    =     \frac12 	\int_\Omega \int_\Omega G_h(x-y) \left| \frac{u(x)-u(y)}{\h}\right|^2\,dx\,dy.
   \end{equation}
   Hence similar to Lemma \ref{lem energy 1}, 
it can be shown to converge to the Dirichlet energy 
\MathSty{\int_\Omega |\nabla u|^2}.
   
We can now follow the proof of Lemma \ref{lemma MM} line by line with 
$u$ and $u^{n-1}$ replaced by $\chi u$ and $\chi u^{n-1}$, respectively, 
where $\chi=\chara_\Omega$. It holds again that 
\eqref{MM Neumann} is equivalent to the 
minimization problem
   \[
   \min\CurBrac{-\frac2h \int \chi u \cdot G_h \ast(\chi u)\,dx,\,\,\,
u:\Omega\to\mathbb{R}^2,\,\,\,\abs{u}\leq 1}.
   \]
    The solution can then be read off as
    \begin{equation}\label{MBO Neumann}
    u^n = \frac{v^n}{|v^n|},\quad \text{where } v^n= G_h\ast (\chi u^{n-1}).
    \end{equation}
Therefore, we propose the following algorithm.
    
    \begin{algo}\Bf{(Neumann boundary conditions)}\label{algo Neumann} 
      Let the initial condition and time-step size be 
$u^0 \colon \Omega \setminus \Gamma_0 \to \R^2$ and $h >0$. 
Given the configuration $u^{n-1}$ at time $t=(n-1)h$, the configuration 
$u^n$ at time $t=nh$ is constructed by the following two operations:
      \begin{enumerate}
       \item Diffusion
       \MathSty{v^n:= G_h\ast (\chara_\Omega u^{n-1})};
       \item Thresholding/Projection onto $\mathbb{S}^1$: 
\MathSty{u^n := \frac{v^n}{|v^n|}} in $\Omega$.
      \end{enumerate}
    \end{algo}
In the above, $\Gamma_0$ is some initial (smooth) curve in $\Omega$. 
Again we require that $u^0$ is well-prepared according to
Definition \ref{def initial}.

    \medskip
    
    Next we consider Dirichlet boundary conditions given by a 
$W^{1,2}_{\text{loc}}$ function $\bar u$ on $ \R^3\setminus \Omega$ with $|\bar u|=1$ a.e.. We start from \eqref{MBO Neumann} and make the following ansatz
    \begin{equation}\label{MBO Dirichlet}
    	u^n := \frac{v^n}{|v^n|},\quad \text{where } v^n= G_h\ast (\chi u^{n-1} +(1-\chi) \bar u),
    \end{equation}
i.e., we essentially set $u$ to be $\bar u$ outside $\Omega$.
    The corresponding minimization problem is then
    \begin{multline}\label{MM Dirichlet}
\min \Big\{  E_{\Omega,h}(u)+ \frac1h \int_\Omega\int_\Omega G_h(x-y) \left(u-u^{n-1}\right)(x) \cdot  \left(u-u^{n-1}\right)(y)\,dx\,dy \\
   + \frac2h \int_\Omega \int_{\R^3\setminus \Omega} G_h(x-y) 
   \left(  1-u(x)\cdot \bar u(y)\right) dy\,dx  \Big\}.
    \end{multline}
    
    The fact that the last term can be interpreted as a penalization can be seen at its asymptotic behavior:
    Given $u, \bar{u} \colon \Omega \to \R^2$ and suppose $\Omega\subset \R^3$ is a smooth bounded domain. Then
    \[
    \lim_{h\to0} \frac1\h \int_\Omega \int_{\R^3\setminus \Omega} G_h(x-y) 
   \left(  1-u(x)\cdot \bar u(y)\right) dy\,dx  
   = \sigma\int_{\partial \Omega}2\left(1-u\cdot \bar u\right)
\,\,\,\,\,\,\Brac{\sigma=\Lover{\sqrt{\pi}}}.
    \]
    
    These asymptotics are not surprising in light of 
\cite[Lemma A.3]{EseOtt14} and \cite[Lemma 2.8]{laux2015convergence}. Indeed, it is shown there
that the measures $\frac1\h \left(1-\chi\right) G_h\ast \chi\,dx$ converge to $\sigma\left|\nabla \chi\right|$ with surface tension $\sigma$, so that for any pair of vector fields $u, \bar u$, 
the leading-order term of the left-hand side is of the form 
    \[\frac1\h\int 2(1-u\cdot \bar u)(1-\chi)G_h\ast \chi\,dx
\,\,\,\text{which converges to}\,\,\,
\sigma  \int 2(1-u\cdot \bar u) \left|\nabla \chi \right| \,\,\, \text{as }h\to 0  .\]

    It is important to point out the difference in the scaling factor in front of the integral in \eqref{MM Dirichlet}.
In particular, if the boundary data are well-prepared, i.e., $|\bar u| = 1$ on $\partial \Omega$, then as $h\downarrow 0$, the third term in \eqref{MM Dirichlet} behaves like
	\[
	\frac{2\sigma}\h \int_{\partial \Omega}\left(1-u\cdot \bar u\right) = \frac\sigma{\h} \int_{\partial \Omega}\left|u- \bar u\right|^2,
	\]
    so that this term is indeed a form of penalization forcing $u$ to assume 
the boundary values $u=\bar u$ on $\partial \Omega$.
    
    Based on the above, we now state the corresponding algorithm for the evolution with Dirichlet boundary conditions.
    \begin{algo}\Bf{(Dirichlet boundary conditions)}\label{algo Dirichlet}
      Let the initial and boundary conditions, and the time-step size be
$u^0 \colon \Omega \setminus \Gamma_0 \to \R^2$ and $\bar u \colon \R^3 \setminus \Omega \to \R^2$, and $h > 0$. 
Given the configuration $u^{n-1}$ at time $t=(n-1)h$, the configuration $u^n$ 
at time $t=nh$ is constructed the following two operations:
      \begin{enumerate}
       \item Diffusion of the by $\bar u$ extended vector field:
       \MathSty{v^n:= G_h\ast (\chara_\Omega u^{n-1} +\chara_{\R^3\setminus \Omega} \bar u)};
       \item Thresholding/Projection onto $\mathbb{S}^1$: 
\MathSty{u^n := \frac{v^n}{|v^n|} \quad \text{in }\Omega}.
      \end{enumerate}

    \end{algo}
    \subsection{Vortex motion with pinning}
      When studying vortices (points) in $\R^2$, the Ginzburg-Landau 
approximation as well as the thresholding scheme discussed above are trivial on the time scale 
under consideration in the sense that the vortices do not move. 
      The easiest way to obtain non-trivial motion goes by the name
``pinning'' which originates from a chemical potential 
$a(x)\geq a_0>0$. Lin \cite{lin1998complex} proved the convergence 
as $\eps\to 0$ of solutions to the equation
      \begin{equation}\label{GL pinning}
        \partial_t u_\eps = \frac1{a(x)} \nabla \cdot \left( a(x) \nabla u_\eps\right) - \frac1{\eps^2} \nabla_u W(u_\eps),
      \end{equation}
      where $W(u) = \frac14(|u|^2-1)^2$.
      The above is the gradient flow of the following weighted energy
      \Beqn\label{GL pinning eng}
	\int a(x)\left(\frac12 |\nabla u|^2 + \frac1{\eps^2} W(u)\right)dx
      \Eeqn
      \Beqn\label{GL pinning diss}
\frac{d}{dt}\int a(x)\left(\frac12 |\nabla u|^2 + \frac1{\eps^2} W(u)\right)dx
= -\int a(x)\abs{\partial_t u}^2.
      \Eeqn
      The motion law for the vortices in the limit $\eps \to0$ is then the ordinary differential equation 
      \[
	\dot X = - \frac{\nabla a(X)}{a(X)}\,\,\Big(= -\nabla \log a(X)\Big),
      \]
      which is again a gradient flow.

      The fact that $a(x)$ arises in both the energy \eqref{GL pinning eng} and 
the dissipation law \eqref{GL pinning diss} gives a hint that thresholding 
can be generalized to this setting as well.
      Indeed, as in the codimension one case, when extending thresholding to 
networks with arbitrary surface tensions, the ``natural'' mobilities 
(in the sense of the scheme) are inversely proportional to the surface tensions \cite{EseOtt14}. 
      
      Now the straight-forward generalization of the minimizing movements 
interpretation \eqref{MM} of thresholding to the vortex motion case is
      \begin{equation}\label{MM pinning}
        \frac1h \int a(x) \left(1-u\cdot G_h\ast u\right)dx + \frac1h \int a(x) \left( u-u^{n-1}\right) \cdot  G_h\ast\left( u-u^{n-1}\right) dx.
      \end{equation}
      Tracing back our steps in the argument for \eqref{MM}, we see that $u^n$ minimizes the linear functional
      \[
       -\frac1h \int  u \cdot \left( a\, G_h \ast u^{n-1} + G_h\ast(a\,u^{n-1})\right)\,dx  
      \]
      among all vector fields $u: \mathbb{R}^2\to\mathbb{R}^2$ with
$\Abs{u}\leq 1$ a.e. Therefore we obtain the following 
variant of thresholding for vortex motion.
      
    \begin{algo}\Bf{(Vortex motion)}\label{algo pinning}
      Let $\{X_1^0,\dots,X_M^0\}\subset ([0, \Lambda)^2)^M$ be the initial locations of vortices.
Let further the initial data and time-step size be
$u^0 \colon [0,\Lambda)^2 \setminus \{X_1^0,\dots,X_M^0\} \to \R^2$
and $h>0$.
Given the configuration $u^{n-1}$ at time $t=(n-1)h$, the configuration 
$u^n$ at time $t=nh$ is constructed by the following two operations:
      \begin{enumerate}
       \item Approximate convection-diffusion process:
       \MathSty{v^n:= G_h\ast (a\,u^{n-1}) + a\, G_h\ast u^{n-1}};
       \item Thresholding/Projection onto $\mathbb{S}^1$:
\MathSty{u^n := \frac{v^n}{|v^n|}}.
      \end{enumerate}
\end{algo}
Note that $v^n$ may be written as
\[
 v^n = 2a\,G_h\ast u^{n-1} + [G_h\ast, a] u^{n-1} \approx 2 a\,G_h \ast u^{n-1} + 2h \nabla a \cdot \nabla G_h \ast u^{n-1},
\]
so that another, yet a priori not energy dissipative, scheme can be obtained by replacing the first step of Algorithm \ref{algo pinning} by 
\[
v^n :=  a\,G_h \ast u^{n-1} + h \nabla a \cdot \nabla G_h \ast u^{n-1}.
\]

We mention in passing that the dynamical law \eqref{GL pinning} can be
changed to the following ``convection-diffusion-reaction'' form
      \Beqn\label{DCR}
        \partial_t u_\eps = \Delta u_\eps + V(x)\cdot\nabla u_\eps
- \frac1{\eps^2} \nabla_u W(u_\eps)
    \Eeqn
where $V$ is some arbitrary (smooth) vector field. 
Even though there is no ``global'' 
variational interpretation for \eqref{DCR} unless $V = \frac{\nabla a}{a}$
for some function $a$, the overall minimizing movements 
strategy and proof of convergence will still work as 
\emph{locally} at each point $x_0$, $V$ can always be approximated
as a gradient of some function. 
More specifically, we simply take $f(x) = \InnProd{V(x_0)}{(x-x_0)}$ 
and $a(x) = e^{f(x)}$ for $x$ near $x_0$. 
Then 
\[
\Abs{V(x) - \frac{\nabla a(x)}{a(x)}}
= \Abs{V(x) - \nabla f(x)} 
= \Abs{V(x) - V(x_0)}
= O(\Abs{x-x_0}). 
\]

For a thresholding scheme to take into account the convection term in \eqref{DCR} we perform an extra
step  in between
the diffusion and thresholding steps in which we deform the ambient domain along the flow generated by the vector field $V$. As long as $V$ has sufficient smoothness, this step 
will at most modify the thresholding energy by a prefactor of 
$1+O(h)$. Hence the overall energy will still remain bounded in finite time.

    \subsection{Harmonic map heat flow in higher dimensions}\label{sec:hmhf}
    The methods used in Section \ref{sec:filament} give a simple proof of convergence for the following time-splitting method for the harmonic map heat flow
    \begin{equation}\label{hmhf}
     \partial_t u -\Delta u = |\nabla u|^2 u
    \end{equation}
    with $u\colon \mathbb{T}^d \to \mathbb{S}^{N-1}$ with initial conditions in $H^1$. 
    
    \begin{algo}\Bf{(Harmonic heat flow)}\label{algo hmhf}
      Let the initial condition and time-step size be
$u^0 \colon \mathbb{T}^d \to \mathbb{S}^{N-1}$ and $h>0$.
Given the configuration $u^{n-1}$ at time $t=(n-1)h$, the configuration 
$u^n$ at time $t=nh$ is constructed by the following two operations:
      \begin{enumerate}
       \item Diffusion: convolve $u^{n-1}$ with the heat kernel, i.e., put $v^n:= G_h\ast u^{n-1}$;
       \item Thresholding/Projection onto $\mathbb{S}^{N-1}$: 
set $u^n := \frac{v^n}{|v^n|}$.
      \end{enumerate}

    \end{algo}
    
    \begin{prop}
    Given initial data $u^0\in H^1([0,\Lambda)^d,\R^N)$ with $|u^0|=1 $ a.e., the (piecewise constant interpolations of the) approximate solutions $u^h$ are pre-compact, i.e., there exists a subsequence $h\downarrow 0$ and a map $u\in H^1([0,\Lambda)^d\times(0,T),\R^N)$ with $|u|\leq 1 $ a.e.\ such that
    \begin{align}
      u^h \to &u & &\text{in } L^2([0,\Lambda)^d \times (0,T))\label{hmhf:comp u},\\
     \nabla G_{h} \ast u^h \rightharpoonup& \nabla u&  &\text{in }L^2([0,\Lambda)^d\times(0,T)) \label{hmhf:comp Du}\quad \text{and}\\
     \partial_t^h\big(G_{h} \ast u^h\big) \rightharpoonup &\partial_t u & &\text{in }L^2([0,\Lambda)^d\times(0,T)).\label{hmhf:comp dtu}
    \end{align}
    Furthermore,  $u$ solves the harmonic map heat flow equation \eqref{hmhf}.
    \end{prop}
    
    \begin{proof}
    Note that the minimizing movements interpretation did not use that fact that the range of $u$ is only two-dimensional. In fact the proof applies line by line in this framework as well.
 	In particular, we have the a priori estimate \eqref{energy estimate}; with the important difference that now 
 	\[
 	E_h(u^0) \leq \int \left|\nabla u^0\right|^2 \quad \text{is uniformly bounded as }h\downarrow0.
 	\]
 	This allows us to prove the compactness statement of the theorem.
 	
    The convergences \eqref{hmhf:comp Du}--\eqref{hmhf:comp dtu} allow us to pass to the limit in the Euler-Lagrange equation \eqref{EL outer weak}, which establishes the equivalent weak form \eqref{HMHF compensated compactness} of \eqref{hmhf}.
    \end{proof}
    
\section{Appendix}
\subsection{Asymptotic expansion of thresholding scheme for filament}\label{appendix asymptotic expansion}
In this section, following \cite{ruuth2001diffusion}, we briefly describe the 
main steps in the asymptotic expansion of the thresholding scheme
demonstrating the appearance of filament curvature motion. Similar
asymptotics for the Ginzburg-Landau dynamics \eqref{GL} is derived 
in \cite{rubinsteinSelfInduced}.

As we are only dealing with the initial conditions $\Gamma_0$, $u^0$, let us omit the index $0$ in the following. 
Denoting $x=(x_1, x_2, x_3)=(x',x_3)\in\mathbb{R}^3$, we work in the geometry 
that the filament can be written as a graph over the $x_3$-axis. Precisely,
let the initial curve $\Gamma$ be given by
\[
\Gamma= \CurBrac{(\gamma_1(x_3), \gamma_2(x_3), x_3): x_3\in\mathbb{R}},
\,\,\,\text{for some smooth functions $\gamma_1$ and $\gamma_2$.}
\]
Identifying $x' =(x_1,x_2) = x_1 + ix_2$ 
and $\gamma= (\gamma_1,\gamma_2) = \gamma_1 + i\gamma_2$,
we use the following as the initial condition for the state variable $v=v(x,t)$:
\[
u(x) = \frac{x'-\gamma(x_3)}
{\abs{x'-\gamma(x_3)}}.
\]
Then the solution at time $t>0$ of the linear heat equation starting
from $u$ is given by
\Beqn\label{heat_conv}
v(x,t) = \Lover{(4\pi t)^{\frac{3}{2}}}
\int_{\mathbb{R}^3}
\exp\Brac{-\frac{\abs{z-x}^2}{4t}}
\frac{z'-\gamma(z_3)}
{\abs{z'-\gamma(z_3)}}
\,dz
\quad\text{where}\quad
z=(z',z_3),\,\,\,t>0.
\Eeqn
We now consider the \Em{outer} and \Em{inner} expansions of the above 
integration.

\Bf{(I) Outer expansion:
$\abs{x'-\gamma(x_3)}\gg \sqrt{t}$.}
For the outer expansion we introduce the length scale $\varepsilon$, the relative coordinate $y$ and the two complex numbers $\zeta$ and $\eta$ given by
\[
\varepsilon:=\sqrt{t},\quad
y := \frac{z - x}{\varepsilon},\quad
\zeta:=x'-\gamma(x_3),\quad
\eta := \varepsilon y' + \gamma(x_3) 
- \gamma(x_3 + \varepsilon y_3).
\]
We observe the following two asymptotics: First, since the parametrization $\gamma$ is smooth
\begin{equation}\label{expand eta}
\eta = \varepsilon(y' - \partial_{x_3}\gamma(y_3)\,y_3) + O(\varepsilon^2 y_3^2).
\end{equation}
Second, for any two complex numbers $\eta$ and $\zeta$
with $\abs{\eta}\ll \abs{\zeta}$ it holds
\[
\frac{\zeta + \eta}{\abs{\zeta + \eta}}
=\frac{\zeta}{\abs{\zeta}}
+ \Brac{\frac{\eta}{\abs{\zeta}} - \frac{\zeta}{\abs{\zeta}}
\text{Re}\Brac{\frac{\eta}{\zeta}}}
+ O\Big(\Abs{\frac{\eta}{\zeta}}^2\Big).
\]
Hence the integral \eqref{heat_conv} can be written as
\begin{eqnarray*}
v(x,t) = \frac{\zeta}{\abs{\zeta}}
+ \Lover{(4\pi)^\frac{3}{2}}\int_{\R^3}
e^{-\frac{\abs{y}^2}{4}}\Big[
\frac{\eta}{\abs{\zeta}} 
- \frac{\zeta}{\abs{\zeta}}\text{Re}\Brac{\frac{\eta}{\zeta}}
\Big]\,dy
+ \Lover{(4\pi)^\frac{3}{2}}\int_{\R^3}
e^{-\frac{\abs{y}^2}{4}}
O\Big(\Abs{\frac{\eta}{\zeta}}^2\Big)
\,dy.
\end{eqnarray*}
Note that the first term of the above expansion is simply $u(x)$.
The second, linear in $\eta$, is in fact almost an expectation
of a Gaussian variable, cf.\ \eqref{expand eta}. So its leading order term vanishes while the contribution form the second order term is $O(\varepsilon^2)$. Hence we conclude that
\Beqn\label{outer-phase-correction}
v(x,t) = u(x) + O(t),
\quad\text{leading to}\quad
\frac{v(x,t)}{\abs{v(x,t)}} = \exp\big(i(\theta + O(t))\big)
\Eeqn
where $\theta$ is the initial phase variable for $u(x)$.

\Bf{(II) Inner expansion:
$\Abs{x' - \gamma(x_3)}\lesssim \sqrt{t}$.} 
Introducing
a constant $\delta \ll 1$ and the new spatial and temporal variables 
$\eta = \frac{x' - \gamma(x_3)}{\delta}$ and 
$\tau = \frac{t}{\delta^2}$, we expand $v(x,t)$ as
\Beqn \label{inner expansion v}
v(x,t) = Ae^{iS} = (A_0 + \delta A_1 + \cdots)e^{i(S_0 + \delta S_1 + \cdots)}.
\Eeqn
Substituting the above into the heat equation $\partial_t v = \Delta v$, we obtain
\begin{eqnarray}
\partial_\tau S - \Delta S + 2\frac{\nabla A}{A}\cdot\nabla S 
+ \delta(\kappa \mathbf{N}-\dot{\Gamma}_t)\cdot\nabla S & = & O(\delta^2),\\
\partial_\tau A - \Delta A + \delta(\kappa \mathbf{N}-\dot{\Gamma}_t)\cdot\nabla A
+ \Abs{\nabla S}^2A & = & O(\delta^2).
\end{eqnarray}
The initial conditions for $S$ and $A$ are $\theta$ and $1$, respectively.

For the $O(1)$ terms in \eqref{inner expansion v} we have
\[
\partial_\tau {S_0} - \Delta S_0 + 2\frac{\nabla A_0}{A_0}\cdot\nabla S_0 = 0,
\quad\text{and}\quad
\partial_\tau {A_0} - \Delta A_0 + \Abs{\nabla S_0}^2A_0 = 0.
\]
The solutions are respectively, $S_0(\eta, \tau) = \theta$ and 
$A_0(\eta, \tau) = A\Brac{\frac{\eta^2}{\tau}}$ 
where the function $A:\mathbb{R}_+\to\mathbb{R}_+$ has the following 
asymptotic behavior:
\[
A(z) \approx \left\{\begin{array}{ll}
z^\Lover{2}(c_0 + c_1z + \cdots), & \text{for $z\ll 1$},\\
e^{-\Lover{z}}(1 + c_2z^{-2} + \cdots), & \text{for $z \gg 1$}.
\end{array}
\right.
\]

Incorporating the $O(\delta)$ term in \eqref{inner expansion v} and making the ansatz that $S$ converges
to a steady state $S_\infty$ as $\eta\to\infty$, we obtain
\Beqn\label{S_eqn_2}
-\Delta S_\infty + \delta(\kappa\mathbf{N} - \dot{\Gamma}_t)\cdot\nabla S_\infty
= O(\delta^2).
\Eeqn
The solution to this equation is given by
\Beqn\label{S_soln_infty}
S_\infty = \eta\int_0^\theta\Big[
G_\eta(\eta, \varphi) 
+ \delta(\kappa\mathbf{N} - \dot{\Gamma}_t)\cdot(\cos\varphi, \sin\varphi)
G(\eta, \varphi)\Big]
\,d\varphi,
\Eeqn
where
\begin{align}\label{Green_formula}
G(\eta,\theta) = & 
-\exp\CurBrac{\frac{\delta}{2}(\kappa\mathbf{N} - \dot{\Gamma}_t)\cdot
(\eta\cos\theta, \eta\sin\theta)}
K_0\Brac{\frac{\delta\eta\abs{\kappa\mathbf{N} - \dot{\Gamma}_t}}{2}},\\
\text{with}\,\,\,
K_0(x) = & - \log\frac{x}{2} + \text{const}\,\,\,
\text{for $x\ll 1$ being the zeroth-order modified Bessel function.}\notag
\end{align}
Substituting \eqref{Green_formula} into \eqref{S_soln_infty}, we obtain
\Beqn\label{inner-exp-angle}
S_\infty = \theta -\frac{3}{2} r
\Big\langle \kappa\mathbf{N} - \dot{\Gamma}_t,
\big(\sin\theta, 1-\cos\theta\big)\Big\rangle
K_0\Big(\frac{r\abs{\kappa\mathbf{N} - \dot{\Gamma}_t}}{2}\Big)
+ O(r\abs{\kappa\mathbf{N} - \dot{\Gamma}_t}),
\Eeqn
where $r = \Abs{x' - \gamma(x_3)}$. 

Now comparing \eqref{inner-exp-angle} 
with \eqref{outer-phase-correction} from the outer expansion, 
we conclude that
\[
\abs{\kappa\mathbf{N} - \dot{\Gamma}_t}
K_0\Big(\frac{\abs{\kappa\mathbf{N} - \dot{\Gamma}_t}}{2}\Big)
+ O(\abs{\kappa\mathbf{N} - \dot{\Gamma}_t}) = O(t),
\]
so that
\[
\dot{\Gamma}_t - \kappa\mathbf{N} = O\Brac{\frac{t}{\abs{\log t}}} = o(t).
\]
Hence the thresholding scheme is consistent for $t \ll 1$. 

Switching back to the notation $h$ for the time step size, 
we note here that the one-step and accumulative errors are respectively,
$O\Big(\frac{h}{\abs{\log h}}\Big)$
and
$O\Big(\Lover{\abs{\log h}}\Big)$.
We defer to future work in making the above asymptotic analysis,
in particular the error estimates, rigorous.

\subsection{Construction of initial conditions}\label{appendix:u0}
 
   We derive the appropriate bounds for the ansatz \eqref{construct u0} for the initial conditions to show they are well-prepared according to Definition \ref{def initial} . As we will only deal with the initial conditions $\Gamma_0$ and $u^0$, we may as well omit the index $0$ in this section. 
   Let $\Gamma \subset \R^3$ be a curve given by
\[
\Gamma = \{(\gamma(x_3),x_3) \colon x_3\in[0,1)\},
\]
where $\gamma = (\gamma_1,\gamma_2) \colon [0,1)\to \R^2$ is a smooth periodic 
vector field. We define
\Beqn\label{appendix construct u0}
u(x) :=  \frac{x'-\gamma(x_3)}{|x'-\gamma(x_3)|}\quad \text{for }x=(x',x_3)\in \R^3\setminus \Gamma.
\Eeqn

By Lemma \ref{lem energy 1} the energy of $u$ can be written as an  average of squared finite differences. More precisely, using \eqref{E finite diff} with $\psi = \chara_\Omega$, we can write the energy in any bounded open set $\Omega \subset [0,\Lambda) \times \R^2$ as
\[
E_h(u,\chara_{\Omega}) =   \frac12  \int_{\R^3} G(z) \int_{\Omega}\left| \frac{u(x) -u(x-\h z)}{\h} \right|^2 dx\, dz.
\]
Now we split the domain of integration in $x$ into a near-field region, which is the $\h$-neighborhood of the filament $A_h := \{x \in \Omega \colon d(x,\Gamma_0)<\h\}$, and its complement, the far-field region. Using $|u|= 1$,  the integral over the near-field region is estimated by
\[
 \frac12\int_{\R^3} G(z) \int_{A_h} \left| \frac{u(x) -u(x-\h z)}{\h} \right|^2  dx\,dz
\leq \frac12\int_{\R^3}  G(z)\int_{A_h}  \frac4h dx\, dz \leq\frac{2}{h} \left|A_h\right|.
 \]
Since the tubular neighborhood $A_h$ of the smooth curve $\Gamma_0$ has Lebesgue measure $|A_h| \leq C h \H^1(\Gamma_0)$ for sufficiently small $h$, the right-hand side is uniformly bounded as $h\downarrow0$.

  The leading-order term of the energy is the integral over the far-field region, which has precisely the asymptotic behavior \eqref{Eu0}. Indeed, the trivial inequality 
  $
  \big| \frac{p}{|p|} - \frac{q}{|q|} \big| \leq 2 \frac{|p-q|}{|p|} 
  $ 
  (which is valid for any two non-zero vectors $p,q$) applied to $p=x'-\gamma(x_3)$ and $q=x'-\h z' - \gamma(x_3-\h z_3)$  implies 
 \[
 \left|\frac{u(x)-u(x-\h z)}{\h}\right|^2
  \leq 4 \frac{|\h z'|^2 + |\gamma(x_3)-\gamma(x_3-\h z_3)|^2}{h|x'-\gamma(x_3)|^2}
 \leq 4 \frac{(1+\|\partial_{x_3} \gamma\|_\infty^2) |z|^2}{|x'-\gamma(x_3)|^2}
 \]
 for all $z\in \R^3$.
 Therefore, the integral over the far-field region is estimated by
 \begin{align*}
	&\frac12 \int_{\R^3} G(z)\int_{\Omega \setminus A_h}  \left| \frac{u(x) -u(x-\h z)}{\h} \right|^2 dx\, dz\\
	&\leq 2(1+\|\partial_{x_3} \gamma\|_\infty^2) \left( \int_{\R^3} |z|^2 G(z)\,dz \right) \int_{\Omega \setminus A_h}\frac{1}{|x'-\gamma(x_3)|^2} dx.
\end{align*}
The prefactor as well ass the Gaussian integral are clearly bounded. If $R<\infty$ is sufficiently  large such that $\Omega \subset B_{R}$ then the last integral can be estimated by
\begin{align*}
\int_{\Omega \setminus A_h}\frac{1}{|x'-\gamma(x_3)|^2} dx 
&\leq \int_0^1 \int_{\{\h < |x'| < 2R\}} \frac{1}{|x'|^2} dx' \,dx_3
=  2\pi \int_{\h}^{2R} \frac1{r^2} r\,dr \\
& = 2\pi  (\log (2R)-\log (\h)) = C(R) (1+ |\log h|).
\end{align*}
This finishes the proof of the upper bound
\[
E_h(u,\chara_\Omega) \leq C |\log h|.
\]
The matching lower bound
\[
E_h(u,\chara_\Omega) \geq c |\log h|
\] 
can be obtained by a  reverse variant of the basic inequality $
  \big| \frac{p}{|p|} - \frac{q}{|q|} \big| \leq 2 \frac{|p-q|}{|p|} 
  $ above, namely $\big| \frac{p}{|p|} - \frac{q}{|q|} \big| \geq \frac1{|p|} \big( \big| p-q\big| - \big| |p|-|q| \big| \big)$.
Furthermore, the uniform bound on the energy away from the filament \eqref{Eu0phi} follows from the derivation of the upper bound above.

When working with several filaments, as for example a periodic pattern of almost parallel filaments, again with periodic boundary conditions in the $x_3$-direction, the vector fields \eqref{appendix construct u0} around each filament with the appropriate choices of the sign may be easily glued together.
We also want to stress that this construction is not restricted to dimension $3$, but applies in general codimension-$2$ surface $\Gamma$ in $\R^d$.

\subsection{Cut-off at infinity}\label{appendix whole space}
When adapting our proof to the whole space $\mathbb{R}^3$ one has to be careful, as the squared gradient as well as our energy densities are not integrable at infinity. 
Note for example that the gradient of the unit vector field $u(x) = \frac {x}{|x|}$ decays with rate $ \frac1{|x|}$, which is not in $L^2(\R^2)$. By slicing it is clear that the behavior for the initial conditions discussed in Appendix \ref{appendix construct u0} is divergent as well.
This can be cured by choosing an appropriate cut-off at infinity. More precisely, the function $\phi_\sigma(x,t)$ given in \eqref{def phi} which we used to localize around the mean curvature flow $\Gamma_t$ can then be replaced by $\phi_\sigma(x,t) \psi_R(|x|)$, where $\psi_R=\psi_R(r)$ is a smooth monotone non-increasing function with $\psi_R(r)=1 $ for $0\leq r\leq R$ and $\psi(r) = \exp(-r)$ for $r\geq 2R$ such that $|\frac{d^k}{dr^k} \psi_R|  \leq C_k\psi_R$ for all $k\in \N$.

\bigskip

\bigskip

\noindent\Bf{Acknowledgement.}
The authors thank Selim Esedo\u{g}lu, Felix Otto, and Drew Swartz for useful 
discussion.
The support by the Purdue Research Foundation and the hospitality of the 
Max Planck Institute for the Mathematical Sciences, Leipzig, Germany, 
are highly noted.

    
\bibliographystyle{plain}

\bibliography{lit}

\begin{thebibliography}{10}

\bibitem{allen1979microscopic}
Samuel~M. Allen and John~W. Cahn.
\newblock A microscopic theory for antiphase boundary motion and its
  application to antiphase domain coarsening.
\newblock {\em Acta Metallurgica}, 27(6):1085--1095, 1979.

\bibitem{ATW93}
Fred Almgren, Jean~E. Taylor, and Lihe Wang.
\newblock Curvature-driven flows: a variational approach.
\newblock {\em SIAM Journal on Control and Optimization}, 31(2):387--438, 1993.

\bibitem{altschuler1992shortening}
Steven~J. Altschuler and Matthew~A. Grayson.
\newblock Shortening space curves and flow through singularities.
\newblock {\em Journal of Differential Geometry}, 35(2):283--298, 1992.

\bibitem{ambrosio2006gradient}
Luigi Ambrosio, Nicola Gigli, and Giuseppe Savar{\'e}.
\newblock {\em Gradient flows in metric spaces and in the space of probability
  measures}.
\newblock Birkh\"auser, 2008.

\bibitem{ambrosio1996level}
Luigi Ambrosio and H.~Mete Soner.
\newblock Level set approach to mean curvature flow in arbitrary codimension.
\newblock {\em Journal of Differential Geometry}, 43:693--737, 1996.

\bibitem{AmbrosioSonerVar}
Luigi Ambrosio and H.~Mete Soner.
\newblock A measure-theoretic approach to higher codimension mean curvature
  flows.
\newblock {\em Annali della Scuola Normale Superiore di Pisa. Classe di
  Scienze. Serie IV}, 25(1-2):27--49, 1997.

\bibitem{barles1995simple}
Guy Barles and Christine Georgelin.
\newblock A simple proof of convergence for an approximation scheme for
  computing motions by mean curvature.
\newblock {\em SIAM Journal on Numerical Analysis}, 32(2):484--500, 1995.

\bibitem{BBHBook}
Fabrice Bethuel, Ha\"{\i}m Brezis, and Fr\'ed\'eric H\'elein.
\newblock {\em Ginzburg-{L}andau vortices}, volume~13 of {\em Progress in
  Nonlinear Differential Equations and their Applications}.
\newblock Birkh\"auser Boston, Inc., Boston, MA, 1994.

\bibitem{bethuel2006convergence}
Fabrice Bethuel, Giandomenico Orlandi, and Didier Smets.
\newblock Convergence of the parabolic {G}inzburg-{L}andau equation to motion
  by mean curvature.
\newblock {\em Annals of Mathematics}, 163(1):37--163, 2006.

\bibitem{bonnetier2012consistency}
Eric Bonnetier, Elie Bretin, and Antonin Chambolle.
\newblock Consistency result for a non monotone scheme for anisotropic mean
  curvature flow.
\newblock {\em Interfaces and Free Boundaries}, 14(1):1--35, 2012.

\bibitem{brakke1978motion}
Kenneth~A. Brakke.
\newblock {\em The motion of a surface by its mean curvature}, volume~20.
\newblock Princeton University Press, Princeton, 1978.

\bibitem{bronsard1991motion}
Lia Bronsard and Robert~V. Kohn.
\newblock Motion by mean curvature as the singular limit of {G}inzburg-{L}andau
  dynamics.
\newblock {\em Journal of Differential Equations}, 90(2):211--237, 1991.

\bibitem{chen1992generation}
Xinfu Chen.
\newblock Generation and propagation of interfaces for reaction-diffusion
  equations.
\newblock {\em Journal of Differential Equations}, 96:116--141, 1992.

\bibitem{chen1991uniqueness}
Yun~G. Chen, Yoshikazu Giga, and Shun’ichi Goto.
\newblock Uniqueness and existence of viscosity solutions of generalized mean
  curvature flow equations.
\newblock {\em Journal of Differential Geometry}, 33(3):749--786, 1991.

\bibitem{chen1989weak}
Yunmei Chen.
\newblock The weak solutions to the evolution problems of harmonic maps.
\newblock {\em Mathematische Zeitschrift}, 201(1):69--74, 1989.

\bibitem{de1993new}
Ennio De~Giorgi.
\newblock New problems on minimizing movements.
\newblock {\em Boundary Value Problems for PDE and Applications}, 29:91--98,
  1993.

\bibitem{de1994barriers}
Ennio De~Giorgi.
\newblock Barriers, boundaries, motion of manifolds.
\newblock volume~18, 1994.

\bibitem{de1995geometrical}
Piero De~Mottoni and Michelle Schatzman.
\newblock Geometrical evolution of developed interfaces.
\newblock {\em Transactions of the American Mathematical Society},
  347(5):1533--1589, 1995.

\bibitem{elsey2016threshold}
Matt Elsey and Selim Esedo\u{g}lu.
\newblock Threshold dynamics for anisotropic surface energies.
\newblock Technical report, UM, 2016.

\bibitem{EseOtt14}
Selim Esedo\u{g}lu and Felix Otto.
\newblock Threshold dynamics for networks with arbitrary surface tensions.
\newblock {\em Communications on Pure and Applied Mathematics}, 68(5):808--864,
  2015.

\bibitem{evans1990weak}
Lawrence~C. Evans.
\newblock {\em Weak convergence methods for nonlinear partial differential
  equations}.
\newblock Number~74. American Mathematical Soc., 1990.

\bibitem{evans1993convergence}
Lawrence~C. Evans.
\newblock Convergence of an algorithm for mean curvature motion.
\newblock {\em Indiana University Mathematics Journal}, 42(2):533--557, 1993.

\bibitem{evans1992phase}
Lawrence~C. Evans, H.~Mete Soner, and Panagiotis~E. Souganidis.
\newblock Phase transitions and generalized motion by mean curvature.
\newblock {\em Communications on Pure and Applied Mathematics},
  45(9):1097--1123, 1992.

\bibitem{EvansSpruck1}
Lawrence~C. Evans and Joel Spruck.
\newblock Motion of level sets by mean curvature i.
\newblock {\em Journal of Differential Geometry}, 33(3):635--681, 1991.

\bibitem{gage1986heat}
Michael Gage and Richard~S. Hamilton.
\newblock The heat equation shrinking convex plane curves.
\newblock {\em Journal of Differential Geometry}, 23(1):69--96, 1986.

\bibitem{HuiskenPoldenGraphProof}
Gerhard Huisken and Alexander Polden.
\newblock Geometric evolution equations for hypersurfaces.
\newblock In {\em Calculus of variations and geometric evolution problems
  ({C}etraro, 1996)}, volume 1713 of {\em Lecture Notes in Math.}, pages
  45--84. Springer, Berlin, 1999.

\bibitem{ilmanen1993convergence}
Tom Ilmanen.
\newblock Convergence of the {A}llen-{C}ahn equation to {B}rakkes motion by
  mean curvature.
\newblock {\em Journal of Differential Geometry}, 38(2):417--461, 1993.

\bibitem{ishiiGeneralRadialKernel}
Hitoshi Ishii.
\newblock A generalization of the {B}ence, {M}erriman and {O}sher algorithm for
  motion by mean curvature.
\newblock In {\em Curvature flows and related topics ({L}evico, 1994)},
  volume~5 of {\em GAKUTO Internat. Ser. Math. Sci. Appl.}, pages 111--127.
  Gakk\=otosho, Tokyo, 1995.

\bibitem{ishii1999threshold}
Hitoshi Ishii, Gabriel~E. Pires, and Panagiotis~E. Souganidis.
\newblock Threshold dynamics type approximation schemes for propagating fronts.
\newblock {\em Journal of the Mathematical Society of Japan}, 51(2):267--308,
  1999.

\bibitem{jerrardICM}
Robert~L. Jerrard.
\newblock Quantized vortex filaments in complex scalar fields.
\newblock In {\em Proceedings of the {I}nternational {C}ongress of
  {M}athematicians---{S}eoul 2014. {V}ol. {III}}, pages 789--810. Kyung Moon
  Sa, Seoul, 2014.

\bibitem{JerrardSonerMCF}
Robert~L. Jerrard and H.~Mete Soner.
\newblock Scaling limits and regularity results for a class of
  {G}inzburg-{L}andau systems.
\newblock {\em Ann. Inst. H. Poincar\'e Anal. Non Lin\'eaire}, 16(4):423--466,
  1999.

\bibitem{laux2015convergence}
Tim Laux and Felix Otto.
\newblock Convergence of the thresholding scheme for multi-phase
  mean-cur\-va\-ture flow.
\newblock {\em Calculus of Variations and Partial Differential Equations},
  55(5):1--74, 2016.

\bibitem{laux2017brakke}
Tim Laux and Felix Otto.
\newblock Brakke's inequality for the thresholding scheme.
\newblock {\em arXiv preprint arXiv:1708.03071}, 2017.

\bibitem{LauxSimon}
Tim Laux and Thilo Simon.
\newblock Convergence of the {A}llen-{C}ahn equation to multiphase mean
  curvature flow.
\newblock {\em To appear in Communications on Pure and Applied Mathematics,
  DOI:10.1002/cpa.21747}, 2018.

\bibitem{LauSwa15}
Tim Laux and Drew Swartz.
\newblock Convergence of thresholding schemes incorporating bulk effects.
\newblock {\em Interfaces and Free Boundaries}, 55(2):273--304, 2017.

\bibitem{lin1998complex}
Fang~Hua Lin.
\newblock Complex {G}inzburg-{L}andau equations and dynamics of vortices,
  filaments, and codimension-2 submanifolds.
\newblock {\em Communications on Pure and Applied Mathematics}, 51(4):385--441,
  1998.

\bibitem{linPanWangHarmonic}
Fanghua Lin, Xing-Bin Pan, and Changyou Wang.
\newblock Phase transition for potentials of high-dimensional wells.
\newblock {\em Comm. Pure Appl. Math.}, 65(6):833--888, 2012.

\bibitem{LucStu95}
Stephan Luckhaus and Thomas Sturzenhecker.
\newblock Implicit time discretization for the mean curvature flow equation.
\newblock {\em Calculus of Variations and Partial Differential Equations},
  3(2):253--271, 1995.

\bibitem{MBO92}
Barry Merriman, James~K. Bence, and Stanley~J. Osher.
\newblock Diffusion generated motion by mean curvature.
\newblock CAM Report 92-18, 1992.
\newblock Department of Mathematics, University of California, Los Angeles.

\bibitem{osting2017generalized}
Braxton Osting and Dong Wang.
\newblock A generalized {MBO} diffusion generated motion for orthogonal
  matrix-valued fields.
\newblock {\em arXiv preprint arXiv:1711.01365}, 2017.

\bibitem{rubinsteinPismen}
L.~M. Pismen and J.~Rubinstein.
\newblock Motion of vortex lines in the {G}inzburg-{L}andau model.
\newblock {\em Physica D. Nonlinear Phenomena}, 47(3):353--360, 1991.

\bibitem{rubinsteinSelfInduced}
Jacob Rubinstein.
\newblock Self-induced motion of line defects.
\newblock {\em Quarterly of Applied Mathematics}, 49(1):1--9, 1991.

\bibitem{rubinstein1989fast}
Jacob Rubinstein, Peter Sternberg, and Joseph~B. Keller.
\newblock Fast reaction, slow diffusion, and curve shortening.
\newblock {\em SIAM Journal on Applied Mathematics}, 49(1):116--133, 1989.

\bibitem{rubinsteinHarmonic}
Jacob Rubinstein, Peter Sternberg, and Joseph~B. Keller.
\newblock Reaction-diffusion processes and evolution to harmonic maps.
\newblock {\em SIAM Journal on Applied Mathematics}, 49(6):1722--1733, 1989.

\bibitem{ruuth2001diffusion}
Steven~J. Ruuth, Barry Merriman, Jack Xin, and Stanley Osher.
\newblock Diffusion-generated motion by mean curvature for filaments.
\newblock {\em Journal of Nonlinear Science}, 11(6):473--493, 2001.

\bibitem{SerfatySandierBook}
Etienne Sandier and Sylvia Serfaty.
\newblock {\em Vortices in the magnetic {G}inzburg-{L}andau model}, volume~70
  of {\em Progress in Nonlinear Differential Equations and their Applications}.
\newblock Birkh\"auser Boston, Inc., Boston, MA, 2007.

\bibitem{SwaYip17}
Drew Swartz and Nung~Kwan Yip.
\newblock Convergence of diffusion generated motion to motion by mean
  curvature.
\newblock {\em Communications in Partial Differential Equations},
  42(10):1598--1643, 2017.

\bibitem{WangMuTaoClassical}
Mu-Tao Wang.
\newblock Long-time existence and convergence of graphic mean curvature flow in
  arbitrary codimension.
\newblock {\em Inventiones Mathematicae}, 148(3):525--543, 2002.

\bibitem{WangMuTaoReview1}
Mu-Tao Wang.
\newblock Mean curvature flows in higher codimension.
\newblock In {\em Second {I}nternational {C}ongress of {C}hinese
  {M}athematicians}, volume~4 of {\em New Stud. Adv. Math.}, pages 275--283.
  Int. Press, Somerville, MA, 2004.

\bibitem{WangMuTaoReview2}
Mu-Tao Wang.
\newblock Lectures on mean curvature flows in higher codimensions.
\newblock In {\em Handbook of geometric analysis. {N}o. 1}, volume~7 of {\em
  Adv. Lect. Math. (ALM)}, pages 525--543. Int. Press, Somerville, MA, 2008.

\end{thebibliography}

    \end{document}